\documentclass{article}

\usepackage{makeidx}
\usepackage{latexsym}
\usepackage{amsfonts}
\usepackage{amssymb}
\usepackage{amsmath}
\usepackage{amstext}
\usepackage{amsthm}
\usepackage{mathrsfs}
\usepackage{color}
\usepackage{geometry}
\usepackage{graphicx}
\usepackage{mathabx}
\usepackage{float}
\usepackage{tikz}
\usetikzlibrary{arrows}
\usepackage{bm}
\usepackage{caption}
\usepackage{subcaption}

\newtheorem{theorem}{Theorem}
\newtheorem{example}[theorem]{Example}

\newtheorem{lemma}[theorem]{Lemma}

\newtheorem{corollary}[theorem]{Corollary}
\newtheorem{definition}[theorem]{Definition}

\newtheorem{proposition}[theorem]{Proposition}

\DeclareMathOperator*{\argmin}{arg\,min}

\begin{document}

\title{Paths and flows for centrality measures in networks}
\author{\textbf{Daniela Bubboloni} \\
{\small {Dipartimento di Matematica e Informatica U.Dini} }\\
\vspace{-6mm}\\
{\small {Universit\`{a} degli Studi di Firenze} }\\
\vspace{-6mm}\\
{\small {viale Morgagni 67/a, 50134 Firenze, Italy}}\\
\vspace{-6mm}\\
{\small {e-mail: daniela.bubboloni@unifi.it}}\\
\vspace{-6mm}\\
{\small https://orcid.org/0000-0002-1639-9525} \and \textbf{Michele Gori}
\\
{\small {Dipartimento di Scienze per l'Economia e  l'Impresa} }\\
\vspace{-6mm}\\
{\small {Universit\`{a} degli Studi di Firenze} }\\
\vspace{-6mm}\\
{\small {via delle Pandette 9, 50127 Firenze, Italy}}\\
\vspace{-6mm}\\
{\small {e-mail: michele.gori@unifi.it}}\\
\vspace{-6mm}\\
{\small https://orcid.org/0000-0003-3274-041X}\\
\vspace{-6mm}\\}

\maketitle

\begin{abstract}
\noindent We consider the number  of paths
that must pass through a subset $X$ of vertices of a network $N$ in a maximum sequence of arc-disjoint paths connecting two vertices $y$ and $z$. We show that when $X$ is a singleton, that number equals the difference between the maximum flow value from $y$ to $z$ in $N$ and the maximum flow value from $y$ to $z$ in the network obtained by $N$ setting to zero the capacities of arcs incident to $X$. That fact theoretically justifies the common identification of those two concepts in network literature. We also show that the same equality does not hold when $|X|\geq 2.$ Consequently, two conceptually different group centrality measures involving paths and flows can naturally be defined, both extending the classic flow betweenness centrality. 
\end{abstract}
\vspace{4mm}

\noindent \textbf{Keywords:} \noindent network flow; arc-disjoint paths; flow betweenness; group centrality.


\section{Introduction}

The concept of flow is certainly one of the most fruitful concepts in network theory with a pletora of recent applications varying from transport engineering to social choice theory and financial networks (see, for instance, Trimponias et al., 2017; Bubboloni and Gori, 2018;  Eboli, 2019). The powerful maxflow-mincut theorem by Ford and Fulkerson (1956) immensely contributed to the success of that concept giving rise, among other things, to the manageable augmenting path algorithm for computing maximum flows. Flows and paths are also at the core of the well-known centrality measure called {\it flow betweenness} due to Freeman et al. (1991). That centrality measure is defined, for a network $N$ with vertex set $V$ and for $x\in V$, by
\begin{equation*}
\label{flow-bet-original}
\Lambda^N(x)=\sum_{\substack{{y,z\in V\setminus\{x\}}\\{y \ne z,}}}\lambda_{yz}^N(x),
\end{equation*}
where $\lambda^N_{yz}(x)$, introduced as ``the maximum flow from $y$ to $z$ that passes through the vertex $x$'' (Freeman et al., 1991, pp.147-148), formally corresponds to the number  of paths that {\it must} pass through $x$ in a maximum sequence of arc-disjoint paths in $N$ connecting $y$ and $z$. The flow betweenness, and some of its variations, are present in UCINET (Borgatti et al. 2002) and in the R sna package (R Development Core Team, 2007). 

Borgatti and Everett (2006, p.475) observe that, since there is not in general a unique maximum sequence of arc-disjoint paths between two vertices, ``flow betweenness cannot be calculated directly by counting paths''. Thus, they explain that in UCINET the number $\lambda_{yz}^N(x)$ is computed as the difference $\varphi^N_{yz}(x)$ of $\varphi_{yz}^N$ and $\varphi_{yz}^{N_{x}}$, where $\varphi_{yz}^N$ is the maximum flow value from $y$ to $z$ in $N$ and $\varphi_{yz}^{N_{x}}$ is the maximum flow value from $y$ to $z$ in the network $N_x$ obtained by $N$ by setting to zero all the capacity on arcs incident to $x$. In other words, $\lambda_{yz}^N(x)$ is replaced by the amount of flow which get lost when all the communications through $x$ are interrupted. The number $\lambda_{yz}^N(x)$ is replaced by $\varphi_{yz}^N(x)$ in the  R sna package  too. In fact, the computation of $\varphi^N_{yz}(x)$ is much less expensive than the one of $\lambda^N_{yz}(x)$ since the two numbers $\varphi_{yz}^N$ and $\varphi_{yz}^{N_{x}}$ can be simply computed via the augmenting path algorithm. 

The identification of $\lambda_{yz}^N(x)$ and $\varphi^N_{yz}(x)$ is indeed very common in the literature (Kosch\"utzki et al., 2005; G\'omez et al., 2013) but, at the best of our knowledge, it is not supported by any rigorous proof. 
The main result of the paper is just the proof that the equality 
\begin{equation}\label{main-equality}
\lambda^N_{yz}(x)=\varphi^N_{yz}(x)
\end{equation}
holds true (Theorem \ref{main}).

Proving \eqref{main-equality} is not a trivial exercise but requires instead a quite sophisticated argument involving some delicate aspects of flow theory and, in particular, the Flow Decomposition Theorem (Theorem \ref{decomposition}). The heart of the matter is that, as shown in detail in Section \ref{proof-main}, it is possible to reconstruct flows from the knowledge of paths and conversely to derive paths from the knowledge of flows. As a consequence, in many situations one can conceptually interchange flows and paths but that interchange is not obvious at all.

In the paper, we also study the natural extensions of $\lambda_{yz}^N(x)$ and $\varphi_{yz}^N(x)$ to the case where sets of vertices are considered. Given a subset $X$ of vertices, we   denote by $\lambda^N_{yz}(X)$ the number of paths that must pass through $X$ in a maximum sequence of arc-disjoint paths connecting two distinct vertices $y$ and $z$. Moreover, we denote by $\varphi^N_{yz}(X)$ the difference of $\varphi_{yz}^N$ and $\varphi_{yz}^{N_{X}}$, where $N_X$ is the network obtained by $N$ by setting to zero all the capacities related to arcs incident to $X$. On the basis of \eqref{main-equality}, one might expect that $\lambda^N_{yz}(X)$ equals $\varphi^N_{yz}(X)$ but that is not true, in general, when $X$ is not a singleton (Proposition \ref{bang}). That fact makes clear that $\lambda^N_{yz}(X)$ and $\varphi^N_{yz}(X)$ are in fact different concepts based on diverse rationales. 

In the last part of the paper, using the quantities $\lambda^N_{yz}(X)$ and $\varphi^N_{yz}(X)$, two new group centrality measures are proposed 
(Section \ref{threecm}). The first one,  based on $\lambda^N_{yz}(X)$ and denoted by $\lambda$, is called {\it full flow  betweenness group centrality measure}; the second one,  based on $\varphi^N_{yz}(X)$ and denoted by $\varphi$, is called {\it full flow vitality group centrality measure}. Those group centrality measures, are inspired to the approach by Freeman (1979), in the sense that they respectively take into consideration the ratios $\lambda^N_{yz}(X)/\lambda^N_{yz}(V)$ and the ratios $\varphi^N_{yz}(X)/\varphi^N_{yz}(V)$.  Some preliminary comments and comparisons on $\lambda$ and $\varphi$ are finally presented.

The paper is organized as follows. In Section \ref{main-definitions} some well-known concepts of network theory are recalled, among which the one of flow, generalized path, path, cycle and sequence of arc-disjoint paths. We propose precise and formal definitions in order to fix notation and allow the proofs to run smoothly. We define then the two main concepts of our research, namely the numbers $\varphi^N_{yz}(X)$ and $\lambda^N_{yz}(X)$. 
In Section \ref{proof-main} we explain how to recover flows from the knowledge of sequences of arc-disjoint paths and conversely. Section \ref{dimo} is devoted to  some instrumental results that allow to produce, in Section \ref{proofs}, the proof of our main results about the numerical relation between $\varphi^N_{yz}(X)$ and $\lambda^N_{yz}(X)$. In Section \ref{further} we study the link among $\varphi^N_{yz}(X)$, $\lambda^N_{yz}(X)$ and the global flow that must pass through $X$ in any maximum flow, interestingly showing that they coincide when $X$ is a singleton.
In Section \ref{threecm} we introduce the two flow group centrality measures $\lambda$ and $\varphi$ and we comment on them. The conclusions close the paper.

\section{Main definitions}\label{main-definitions}

\subsection{Notation and preliminary definitions}\label{sec-2}

Throughout the paper, $\mathbb{N}$ denotes the set of positive integers and
$\mathbb{N}_0= \mathbb{N}\cup\{0\}.$  If $m\in \mathbb{N}_0$ we set $[m]=\{n\in \mathbb{N}: n\leq m\}$. In particular, $[0]=\varnothing$ and $|[m]|=m$ for all $m\in \mathbb{N}_0$.  As usual, the sum of real numbers (of real valued functions) over an empty set of indices is assumed to be the number $0$ (the constant function $0$). 

Let $X$ be a (possibly empty) set and $m\in \mathbb{N}$. A sequence of $m$ elements in $X$ is an element of the cartesian product $X^m$. Given $\bm{x}=(x_j)_{j\in [m]}=(x_1,\ldots,x_m)\in X^m$ and $j\in [m]$, we say that $x_j\in X$ is the $j$-th component of $\bm{x}$. Of course, different components of the same sequence can be equal. Note that $X^m\neq \varnothing$ if and only if $X\neq \varnothing$ so that there are sequences of $m$ elements in $X$ if and only if $X\neq\varnothing$.
We also set $X^0=\{()\}$ and call the symbol $()$ the sequence of $0$ elements of $X$. In order to have a uniform notation for sequences of $0$ elements of $X$ and sequences of $m\ge 1$ elements of $X$, we will always interpret as $()$ any writing of the type $(x_j)_{j\in [0]}$. 
Finally, given two sequences of elements of $X$, we say that they are {\it equivalent} if they both have $0$ elements or if they have the same number of elements and one can be obtained from the other by a permutation of the components.

Let $V$ be a finite set with $|V|\ge 2$. The complete digraph on $V$ is the digraph $K_V=(V,A)$ with vertex set $V$ and arc set $A=\{(x,y)\in V^2: x\neq y\}$. 
Note that in a complete digraph the set $A$ of arcs is completely determined by the choice of the set $V$ of vertices.
The set of the complete digraphs is denoted by $\mathscr{K}$. A network is a pair $N=(K_V,c)$, where $K_V=(V,A)\in \mathscr{K}$ and $c$ is a function 
from $A$ to $\mathbb{N}_0$ called capacity. If convenient we will also indicate a network in a more detailed way by $N=(V,A,c)$.
The set of networks is denoted by $\mathscr{N}$.\footnote{To avoid insidious set theory issues one can, of course, assume $V\subseteq \mathbb{N}$.}

\subsection{Flows in a network}

Let $N=(K_V,c)\in \mathscr{N}$ be fixed in the rest of this section. If $a=(x,y)\in A$ we call $x$ and $y$ the endpoints of $a$. Moreover, we say that $a$ exits from $x$ and enters in $y$ . Let $X\subseteq V$. An arc $a\in A$ is called incident to $X$ if at least one of its endpoints belongs to $X$.
We define 
\begin{equation*}\label{archi}
A_X^+:=\{(x,u)\in A: x\in X, u\in V\setminus X\},\quad A_X^-:=\{(u,x)\in A: x\in X, u\in V\setminus X\}, \quad A_X:= A_X^+\cup A_X^-.
\end{equation*}
Note that $A_X$ is the set of arcs in $A$ with a unique endpoint belonging to $X$.\footnote{Throughout the paper, in all the writings involving a subset $X$ of the set of vertices of a network, we write $x$ instead of $X$ when $X=\{x\}$, for some vertex $x$. Thus, for instance, we write $A^+_x$ instead of $A^+_{\{x\}}$.}
We define by $N_X$  to be the network $(K_V,c_X)\in \mathscr{N}$ where, for every $a\in A$,
\[
c_X(a):=\left\{
\begin{array}{ll}
0 & \mbox{if}\ a\mbox{ is incident to } X\\
c(a) & \mbox{otherwise}.
\end{array}
\right.
\]
Within flow theory, the capacity of $X$ is defined by $c(X):=\sum_{a\in A_X^+}c(a).$
Note that if $x\in V$, then $c(x)$ is the so-called outdegree of $x$ while
$c(V\setminus\{x\})$ is the so-called indegree of $x$.

Let $y,z\in V$ be distinct. Recall that a flow from $y$ to $z$ in $N$ is a function $f:A\to \mathbb{N}_0$ such that, for every $a\in A$,
\begin{equation}\label{compatibility}
0\le f(a)\le c(a) \quad \hbox{(compatibility)}
\end{equation}
and, for every $x\in V\setminus\{y,z\}$,
\begin{equation}\label{conservation}
\sum_{a\in A_x^-}f(a)=\sum_{a\in A_x^+}f(a)\quad \hbox{(conservation law)}.
\end{equation}
The function $f_0:A\to \mathbb{N}_0$ defined by $f_0(a)=0$ for all $a\in A$ is a flow, called the  null flow.
We denote the set of flows from $y$ to $z$ in $N$ by $ \mathcal{F}(N,y,z)$.

When we represent networks and flows via a figure, we are going to use some standard conventions: a single number attached to an arc represents the capacity of that arc; two numbers attached to an arc respectively represent the flow and the capacity of that arc; if an arc is not drawn, then its capacity is zero. 

Recall that, given $f\in  \mathcal{F}(N,y,z)$,  the  value of $f$ is the non-negative integer
\[
v(f):=\sum_{a\in A_y^+}f(a)-\sum_{a\in A_y^-}f(a).
\]
The number
\[
\varphi^N_{yz}:=\max_{f\in  \mathcal{F}(N,y,z)} v(f),
\]
is called the maximum flow value from $y$ to $z$ in $N$. 
If $f\in  \mathcal{F}(N,y,z)$ is such that $v(f)=\varphi^N_{yz}$, then $f$ is called a  maximum flow from $y$ to $z$ in $N$. We denote the set of maximum flows from $y$ to $z$ in $N$ by $ \mathcal{M}(N,y,z)$. 

Given $N'=(K_V,c')\in\mathscr{N}$ with $c'\le c$, it is immediate to observe  that 
\begin{equation}\label{utile-inclusione}
\mathcal{F}(N',y,z)\subseteq \mathcal{F}(N,y,z),
\end{equation}
and 
\begin{equation}\label{utile-disu}
\varphi^{N'}_{yz}\le \varphi^{N}_{yz}.
\end{equation}
Let us introduce now an important definition. 

\begin{definition}\label{flow-per-x}
Let $f\in \mathcal{F}(N,y,z)$. For every $x\in V$, we set
\begin{equation*}\label{fx}
f(x):=
\left\{
\begin{array}{ll}
\displaystyle\sum_{a\in A_x^+}f(a) &\mbox{ if }x\not\in \{y,z\}\\
\vspace{-2mm}\\
v(f) &\mbox{ if }x\in \{y,z\}.
\end{array}
\right.
\end{equation*}
For every $X\subseteq V$, we next set
\begin{equation*}\label{fX}
f(X):=\sum_{x\in X}f(x)
\end{equation*}
and we call $f(X)$ the flow that passes through $X$ in the flow $f$. 
\end{definition}
Note that $f(X)\ge 0$ and that  if $X\cap \{y,z\}\neq \varnothing$, then $f(X)\geq v(f)$.

\subsection{The number $\varphi_{yz}^N(X)$}

Let us introduce now the first main concept of our research, namely the number $\varphi^N_{yz}(X)$.

\begin{definition} \label{phi}
Let  $N=(K_V,c)\in \mathscr{N}$, $y,z\in V$ be distinct and $X\subseteq V$. We define
\[
\varphi^N_{yz}(X):=\varphi^N_{yz}-\varphi^{N_X}_{yz},
\]
\end{definition}

The number $\varphi^N_{yz}(X)$ represents the falling of maximum flow value from $y$ to $z$ in $N$ when the capacity of all the arcs incident to $X$ are set to zero. Note that $X\cap \{y,z\}\neq\varnothing$ implies $\varphi^{N}_{yz}(X)=\varphi^{N}_{yz}$. The next proposition states a useful monotonicity property of $\varphi^{N}_{yz}(X)$.

\begin{proposition}\label{monot-phi}
Let  $N=(K_V,c)\in \mathscr{N}$, $y,z\in V$ be distinct and $X\subseteq Y\subseteq V$. Then 
\[
0\le \varphi^N_{yz}(X)\le \varphi^N_{yz}(Y).
\]
\end{proposition}
\begin{proof}
By Definition \ref{phi}, we have that
$\varphi^N_{yz}(X)= \varphi^N_{yz}-\varphi^{N_X}_{yz}$ and $\varphi^N_{yz}(Y)= \varphi^N_{yz}-\varphi^{N_Y}_{yz}$. Since $X\subseteq Y$ we have that 
$N_X=(K_V,c_X)$ and $N_Y=(K_V,c_Y)$ are such that $c_Y\le c_X\le c$. Thus, by \eqref{utile-disu}, we have that
$\varphi^{N_Y}_{yz}\le \varphi^{N_X}_{yz}\le \varphi^{N}_{yz}$ which in turn implies $0\le \varphi^N_{yz}(X)\le \varphi^N_{yz}(Y)$, as desired.
\end{proof}

\subsection{Generalized paths and cycles in a complete digraph} \label{paths and cycles}

Let $K_V\in\mathscr{K}$ and $y,z\in V$ be distinct. Consider a pair $\gamma=((x_1,\ldots,x_m),(a_1,\ldots,a_{m-1}))$, where $m\geq 2$, $x_1,\ldots,x_m\in V$ are called the vertices of $\gamma$, $a_1,\ldots,a_{m-1}\in A$ are called the arcs of $\gamma$. The set of vertices of $\gamma$ is denoted by $V(\gamma)$ and the set of arcs  by $A(\gamma)$. Given $X\subseteq V$, we say that 
$\gamma$ passes through $X$ if $X\cap V(\gamma)\neq \varnothing$. We are interested in the following specifications for $\gamma$.
\begin{itemize}
\item[1.] $\gamma$ is called  a {\it generalized path}  in $K_V$ if $x_1,\ldots,x_m$ are distinct and, for every $i\in [m-1]$, $a_i=(x_i,x_{i+1})$ or $a_i=(x_{i+1},x_{i})$. If $a_i=(x_i,x_{i+1})$, $a_i$ is called a {\it forward arc}; if $a_i=(x_{i+1},x_{i})$, $a_i$ is called a {\it backward arc}.  Note that, as a consequence, $a_1,\ldots,a_{m-1}$ are distinct too.
The set of forward arcs  is denoted by $A(\gamma)^+$; the set of backward arcs  by $A(\gamma)^-$. Clearly,  we have $A(\gamma)=A(\gamma)^+\cup A(\gamma)^-$ and $A(\gamma)^+\cap A(\gamma)^-=\varnothing$. We say that $\gamma$ is a generalized path from $y$ to $z$ if $x_1=y$ and $x_m=z$.
\item[2.] $\gamma$ is called a {\it path} in $K_V$ if  $\gamma$ is a generalized path and $A(\gamma)^-=\varnothing$ or, equivalently $A(\gamma)=A(\gamma)^+$.
\item[3.] $\gamma$ is called a {\it cycle} in $K_V$ if $m\ge 3$, $x_1,\ldots,x_{m-1}$ are distinct elements of $V$ while $x_m=x_1$ and, for every $i\in \{1,\ldots,m-1\}$, $a_i=(x_i,x_{i+1})$.
\end{itemize}
Let $\gamma=((x_1,\ldots,x_m),(a_1,\ldots,a_{m-1}))$ be a path or a cycle in $K_V$. Then $\gamma$ is completely determined by its vertices and thus we usually write $\gamma=x_1\cdots x_m.$ Of course, the same simple notation is not possible for generalized paths that are not paths.

\subsection{Arc-disjoint sequences of paths in a network}\label{intro-lambda}

Let $N=(K_V,c)\in \mathscr{N}$. A path (cycle) $\gamma$ in $K_V$ is called a path (cycle) in $N$ if, for every arc $a\in A(\gamma)$, $c(a)\geq 1$.
The set of paths from $y$ to $z$ in $N$ is denoted by $P^N_{yz}$. The set of cycles in $N$ is denoted by $C^N$. We give no meaning to generalized paths in $N$. 

\begin{definition}\label{arc-disj-def}
Given $m\in\mathbb{N}_0$, a sequences of $m$ paths $\bm{\gamma}=(\gamma_{j})_{j\in[m]}\in (P^N_{yz})^m$
is called {\it arc-disjoint}  if, for every $a\in A$,
\begin{equation}\label{ind-flow-eq}
|\{j\in [m]: a\in A(\gamma_{j})\}|\le c(a).
\end{equation}
\end{definition}
Trivially, if $\bm{\gamma}$ is arc-disjoint and $\bm{\gamma}'$ is equivalent to $\bm{\gamma}$, then $\bm{\gamma}'$ is arc-disjoint too.
We denote the set of sequences of $m$ arc-disjoint paths from $y$ to $z$ in $N$ by $\mathcal{S}^{N,m}_{yz}$. Note that $\mathcal{S}^{N,0}_{yz}=\{(\,)\}$ and that
$P^N_{yz}=\emptyset$ implies $\mathcal{S}^{N,m}_{yz}=\emptyset$ for all $m\geq 1$.
The set of sequences of arc-disjoint paths from $y$ to $z$ in $N$ is defined by
\begin{equation*}
\mathcal{S}_{yz}^{N}:=\displaystyle\bigcup_{m\in\mathbb{N}_0}\mathcal{S}^{N,m}_{y,z}.
\end{equation*}
Note that, since $(\,)\in \mathcal{S}_{yz}^{N}$, we always have $\mathcal{S}_{yz}^{N}\ne\varnothing$. Moreover, $P^N_{yz}=\varnothing$ if and only if $\mathcal{S}_{yz}^{N}=\{(\,)\}$.
If $\bm{\gamma}\in \mathcal{S}_{yz}^{N,m}$, we say that the length of $\bm{\gamma}$ is $m$ and we write $l(\bm{\gamma})=m$.
Observe that, $l(\bm{\gamma})=0$ if and only if $\bm{\gamma}=(\,)$.
We also set
\begin{equation*}
\label{eqn:lambdast}
\lambda^N_{yz}:=\max\left\{m\in\mathbb{N}_0:\,\mathcal{S}_{yz}^{N,m}\ne \emptyset\right\}.
\end{equation*}
Note that $\lambda^N_{yz}$ is  the maximum length of a sequence of arc-disjoint paths from $y$ to $z$ in $N$ and that $\lambda^N_{yz}=0$ if and only if $P^N_{yz}=\varnothing$. The set of sequences of arc-disjoint paths from $y$ to $z$ in $N$ having maximum length is defined by
\begin{equation*}
\label{eqn:maxsetarc-disjointpaths}
\mathcal{M}_{yz}^{N}:=\mathcal{S}^{N,\lambda^{N}_{yz}}_{yz}=\left\{\bm{\gamma}\in\mathcal{S}^{N}_{yz}:\,l(\bm{\gamma})=\lambda^{N}_{yz}\right\}.
\end{equation*}
By Lemma 7.1.5 in Bang-Jensen and Gutin (2008), we know that
\begin{equation}
\label{lex:disjointpaths}
\varphi^N_{yz}=\lambda^N_{yz}.
\end{equation}
Hence, we also have $\mathcal{M}_{yz}^{N}=\mathcal{S}^{N,\varphi^{N}_{yz}}_{yz}$ so that if $\bm{\gamma}\in\mathcal{M}^{N}_{yz}$,
then $l(\bm{\gamma})=\varphi^{N}_{yz}$.

We emphasize that, given $\bm{\gamma}=(\gamma_{j})_{j\in[m]}\in \mathcal{S}^{N,m}_{yz}$, where $m\in \mathbb{N}$ and $m<n\leq \varphi^{N}_{yz}$ (so that $\bm{\gamma}$ has not maximum length), then it is not generally guaranteed that there exists a sequence $(\gamma_{j})_{j\in\{m+1,\dots, n\}}$ of  $n-m$ arc-disjoint paths from $y$ to $z$ in $N$ such that  $(\gamma_{j})_{j\in[n]}\in  \mathcal{S}^{N,n}_{yz}$. In other words, one cannot generally add paths to a sequence of arc-disjoint paths to get a new sequence of arc-disjoint paths of higher length. That fact surely introduces an element of complexity  in treating the sequences of  arc-disjoint paths. We will discuss in more detail that issue after having presented
the Flow Decomposition Theorem (Theorem \ref{decomposition}) in Section \ref{flows-vs-paths}.

Given $\bm{\gamma}\in\mathcal{S}_{yz}^{N}$ and $X\subseteq V$, we denote now by $l_X(\bm{\gamma})$ the number of components of $\bm{\gamma}$ passing through $X$. Formally,
if $\bm{\gamma}=(\gamma_j)_{j\in [m]}\in\mathcal{S}_{yz}^{N}$, where $m\in\mathbb{N}_0$, we set
\[
l_X(\bm{\gamma}):=|\{j\in [m]: \gamma_j\mbox{ passes through }X\}|.
\]
Note that if $\bm{\gamma}, \bm{\gamma}' \in \mathcal{S}_{yz}^{N}$ are equivalent, then $l_X(\bm{\gamma})=l_X(\bm{\gamma}').$

\subsection{The number $\lambda^N_{yz}(X)$ }\label{main-def}

We now have all the tools for providing the definition of the other main concept of our research, namely the number $\lambda^N_{yz}(X)$.\footnote{When $X$ is a singleton, a definition similar to Definition \ref{lambda} was proposed in Ghiggi (2018).} 

\begin{definition} \label{lambda}
Let  $N=(K_V,c)\in \mathscr{N}$, $y,z\in V$ be distinct and $X\subseteq V$. We define
\[
\lambda^{N}_{yz}(X):= \min_{\bm{\gamma}\in\mathcal{M}_{yz}^{N}}l_X(\bm{\gamma}).
\]
\end{definition}
 
The number $\lambda^N_{yz}(X)$ represents the number of paths
that must pass through  $X$ in a maximum sequence of arc-disjoint paths connecting the vertices $y$ and $z$.
Since $\mathcal{M}_{yz}^{N}\ne\varnothing$, $\lambda^{N}_{yz}(X)$ is well defined. 
Note that $X\cap \{y,z\}\neq \varnothing$ implies $\lambda^N_{yz}(X)= \varphi^N_{yz}$. 
Moreover, given $X\subseteq Y\subseteq V$, it is immediate to prove that $0\le \lambda^N_{yz}(X)\le \lambda^N_{yz}(Y)$.
Finally, $\lambda^{N}_{yz}(X)=0$ if and only if there exists $\bm{\gamma}\in\mathcal{M}^{N}_{yz}$ such that $l_X(\bm{\gamma})=0$, that is, none of the paths appearing as components of  $\bm{\gamma}$ passes through $X$.  

We also set
\begin{equation*}
\label{eqn:mpnx}
\mathcal{M}^{N}_{yz}(X):=\displaystyle\argmin_{\bm{\gamma}\in\mathcal{M}^{N}_{yz}}l_X(\bm{\gamma}).
\end{equation*}
Note that $\mathcal{M}^{N}_{yz}(X)\ne\varnothing$ and that if $\bm{\gamma}\in\mathcal{M}^{N}_{yz}(X)$, then $l_X(\bm{\gamma})=\lambda^{N}_{yz}(X)$ and $l(\bm{\gamma})=\varphi^{N}_{yz}.$
In other words, the set  $\mathcal{M}^{N}_{yz}(X)$ collects the maximum sequences of arc-disjoint paths from $y$ to $z$ in $N$ minimally passing through $X.$

\begin{example}\label{figura1}{\rm
In order to clarify Definition \ref{lambda}, let us perform some explicit computations for the network $N$ in Figure \ref{example1}.
\begin{figure}[t]
\begin{center}
\begin{tikzpicture}
\begin{scope}
    \node[style={circle,thick,draw,fill=white}] (y) at (0,0) {$y$};
    \node[style={circle,thick,draw,fill=lightgray}] (a) at (2,1) {$v$};
    \node[style={circle,thick,draw,fill=lightgray}] (c) at (4,1) {$x$};
    \node[style={circle,thick,draw,fill=lightgray}] (b) at (3,-1){$u$};
    \node[style={circle,thick,draw,fill=white}] (z) at (6,0) {$z$} ;
\end{scope}
\begin{scope}[>=stealth, every edge/.style={thick,draw}]    \path [->] (y) edge node[pos=0.5,anchor=south]  {$2$} (a);
    \path [->] (y) edge node[pos=0.5,anchor=north]  {$1$} (b);
    \path [->] (a) edge node[pos=0.5,anchor=south]  {$1$} (c);
    \path [->] (a) edge node[pos=0.5,anchor=east]  {$1$} (b);
		\path [->] (b) edge node[pos=0.5,anchor=west]  {$2$} (c);
		\path [->] (b) edge node[pos=0.5,anchor=north]  {$1$} (z);
		\path [->] (c) edge node[pos=0.5,anchor=south]  {$2$} (z);
   
\end{scope}
\end{tikzpicture}
\end{center}
\caption{}
\label{example1}
\end{figure}
First of all, we have that $P^{N}_{yz}=\{\gamma_1,\gamma_2,\gamma_3,\gamma_4,\gamma_5\}$, where 
$$\gamma_1=yvuxz, \ \gamma_2=yvuz, \ \gamma_3=yvxz, \ \gamma_4=yuxz, \ 
\gamma_5=yuz.$$
Of course, $\mathcal{S}^{N,0}_{yz}=\{()\}$ and $\mathcal{S}^{N,1}_{yz}=\{(\gamma_1),(\gamma_2),(\gamma_3),(\gamma_4),(\gamma_5)\}$.
Up to a reordering of the components, the elements of $\mathcal{S}^{N,2}_{yz}$ are given by 

\[
(\gamma_1,\gamma_3),\;(\gamma_1,\gamma_5),\;(\gamma_2,\gamma_3),\;(\gamma_2,\gamma_4),\;(\gamma_3,\gamma_4),\;(\gamma_3,\gamma_5),
\]
while the elements of $\mathcal{S}^{N,3}_{yz}$ are given by 

\[
\bm{\gamma}'=(\gamma_1,\gamma_3,\gamma_5),\  \bm{\gamma}''=(\gamma_2,\gamma_3,\gamma_4).
\]
Moreover, for every $m\ge 4$, $\mathcal{S}^{N,m}_{yz}=\varnothing$. As a consequence, $\lambda^{N}_{yz}=3$ and $\mathcal{M}_{yz}^{N}=\mathcal{S}^{N,3}_{yz}$.
Considering now $X=\{x\}$, we have that $l_X(\bm{\gamma}')=2$ and $l_X(\bm{\gamma}'')=2$. Thus, 
\[
\lambda^{N}_{yz}(X)=\min_{\bm{\gamma}\in \mathcal{M}_{yz}^{N}}l_X(\bm{\gamma})=\min\{l_X(\bm{\gamma}'), l_X(\bm{\gamma}'')\}=2,
\]
and $\mathcal{M}_{yz}^{N}(X)=\mathcal{M}_{yz}^{N}$. Considering instead $X=\{x,v\}$, 
we have that $l_X(\bm{\gamma}')=2$ and $l_X(\bm{\gamma}'')=3$. Thus,
\[
\lambda^{N}_{yz}(X)=\min_{\mathbf{\gamma}\in \mathcal{M}_{yz}^{N}}l_X(\mathbf{\gamma})=\min\{l_X(\bm{\gamma}'), l_X(\bm{\gamma}'')\}=2,
\]
and the elements of $\mathcal{M}_{yz}^{N}(X)$ are given by $\bm{\gamma}''$ and all the sequences equivalent to $\bm{\gamma}''$.}
\end{example}

We close this section with a final comment about some misunderstandings that appeared in the literature when $X=\{x\}$.
Newman (2005, p.41, note 3), citing Freeman et al. (1991) about their description of $\lambda^{N}_{yz}(x)$, explains that, in order to take into account the fact that there is, in general, more than one maximum sequence of arc-disjoint paths from $y$ to $z$ in $N$, they consider ``the maximum possible flow through $x$ over all possible solutions to the $yz$ maximum flow problem''. Within our notation that means to consider the quantity $\max_{\bm{\gamma}\in\mathcal{M}_{yz}^{N}}l_X(\bm{\gamma})$. The idea to take the maximum, instead of the minimum as in our Definition \ref{lambda}, does not seem in line with the spirit of the original definition by Freeman et al. (1991).  Indeed, $\max_{\bm{\gamma}\in\mathcal{M}_{yz}^{N}}l_X(\bm{\gamma})$ describes the flow that {\it can} pass through $x$, and not the one that {\it must} pass through $x$, in any maximum flow.\footnote{A similar problem seems to be present in the description of $\lambda^{N}_{yz}(x)$ in the recent book by Zweig (2016, p.253).}

\section{Paths and flows}\label{proof-main}

Once the formal definitions of $\varphi_{yz}^N(X)$ and $\lambda_{yz}^N(X)$ are given, our main purpose is to analyse the relation between those numbers. It turns out fundamental to deepen the link between paths and flows. A careful description of that link constitutes the indispensable tool for the proof of our main theorem, namely Theorem \ref{main}. First of all, let us introduce the concepts of generalized path function, path function and cycle function.

\begin{definition}\label{path-function} Let  $K_V\in\mathscr{K}$. If $a\in A$, the arc function associated with $a$ is the function
$\chi_a:A\rightarrow \mathbb{N}_0$ defined by $\chi_a(a)=1$ and $\chi_a(b)=0$ for all $b\in A\setminus\{a\}$.
If $\gamma$ is a generalized path in $K_V$, let $\chi_{\gamma}:A\rightarrow \mathbb{Z}$ be defined by 
\begin{equation}\label{chi-generalizzato}
\chi_{\gamma}:=\sum_{a\in A(\gamma)^+} \chi_{a}- \sum_{a\in A(\gamma)^-} \chi_{a},
\end{equation}
so that, if $\gamma$ is a path, then 
\begin{equation}\label{chi-path}
\chi_{\gamma}=\sum_{a\in A(\gamma)}\chi_{a}.
\end{equation}
If $\gamma$ is a cycle in $K_V$, let $\chi_{\gamma}:A\rightarrow \mathbb{Z}$ be defined by
\begin{equation}\label{chi-cycle}
\chi_{\gamma}:=\sum_{a\in A(\gamma)}\chi_{a}.
\end{equation}
The functions \eqref{chi-generalizzato}, \eqref{chi-path}, \eqref{chi-cycle} are respectively called the generalized path function, the path function and the  cycle functions associated with $\gamma.$
\end{definition}
Note that the generalized path functions assume values in $\{-1,0,1\}$. 
In particular, given a generalized path $\gamma$, we have that $\chi_{\gamma}(a)=-1$ if and only if $a$ is 
a backward arc of $\gamma$. Path functions 
and cycle functions assume instead only values in $\{0,1\}$

\subsection{From paths to flows }\label{paths-vs-flows}

The next result shows how every sequence of  arc-disjoint  paths can define a flow.

\begin{proposition}
\label{ind-flow}
Let $N=(V,A,c)\in \mathscr{N}$, $y,z\in V$ be distinct and $\bm{\gamma}=(\gamma_j)_{j\in[m]}\in \mathcal{S}_{yz}^{N,m}$, for some $m\in \mathbb{N}_0$.
Then the function
$f_{\bm{\gamma}}:A\rightarrow \mathbb{N}_0$ defined, for every $a\in A,$ by
\begin{equation}\label{ind-flow-def}
f_{\bm{\gamma}}(a)=|\{j\in [m]: a\in A(\gamma_{j})\}|
\end{equation}
is a flow from $y$ to $z$ in $N$ with $v(f_{\bm{\gamma}})=m$ and $f_{\bm{\gamma}}= \sum_{j\in[m]} \chi_{\gamma_j}.$
Moreover, for every $x\in V,$ we have
\begin{equation}\label{fgammax}
f_{\bm{\gamma}}(x)=l_x(\bm{\gamma}).
\end{equation}
\end{proposition}

\begin{proof} By \eqref{ind-flow-eq}, we immediately obtain that $f_{\bm{\gamma}}$ satisfies the compatibility condition \eqref{compatibility}.
For every $a\in A$, define $U_a=\{j\in [m]: a\in A(\gamma_{j})\}$. Note that $U_a\subseteq [m]$ and $|U_a|=f_{\bm{\gamma}}(a)$.

Let $x\in V$ and $a,b\in A$ with $a\neq b$. If $a,b\in A_x^-$ or if $a,b\in A_x^+$,
then we have
\begin{equation}\label{empty}
U_a\cap U_b=\varnothing.
\end{equation}
Indeed, let $a,b\in A_x^-$ and assume by contradiction that there exists $j\in U_a\cap U_b.$ Then both $a$ and $b$ are arcs of the path $\gamma_j$ entering into its vertex $x$, against the fact that in a path every vertex has at most one arc entering into it. The same argument applies to the case $a,b\in A_x^+$.

Consider now $x\in V\setminus\{y,z\}$. It is immediately checked that
\begin{equation}\label{prelim1}
\bigcup _{a\in A_x^-}U_a=\{j\in[m]: \gamma_j \mbox{ passes through }x\}=\bigcup _{a\in A_x^+}U_a,
\end{equation}
Then, using \eqref{empty} and \eqref{prelim1}, we get
\begin{equation}\label{cons2}
\sum_{a\in A_x^-}f_{\bm{\gamma}}(a)=\sum_{a\in A_x^-}|U_a| =\left|\bigcup _{a\in A_x^-}U_a\right|=l_x(\bm{\gamma})=\left|\bigcup _{a\in A_x^+}U_a\right|=
\sum_{a\in A_x^+}|U_a|=
\sum_{a\in A_x^+}f_{\bm{\gamma}}(a),
\end{equation}
which says that $f_{\bm{\gamma}}$ satisfies the conservation law \eqref{conservation}. Thus, we have proved that $f_{\bm{\gamma}}$ is a flow. By \eqref{cons2}, we  also see that
$$f_{\bm{\gamma}}(x)=
\sum_{a\in A_x^+}f_{\bm{\gamma}}(a)
=l_x(\bm{\gamma}).$$

We next show that
\begin{equation}\label{cons3}
\bigcup _{a\in A_y^+}U_a=[m].
\end{equation}
We surely have $\bigcup _{a\in A_y^+}U_a\subseteq [m]$, so that we are left with proving $ [m]\subseteq\bigcup _{a\in A_y^+}U_a$.
If $m=0$, then $[m]=[0]=\varnothing$ and the desired inclusion immediately holds. Assume next that $m\geq 1.$ Pick $j\in [m]$ and consider $\gamma_j$. Since $y\neq z$, there exists $a\in A(\gamma_j)\cap A_y^+$ and therefore $j\in \bigcup _{a\in A_y^+}U_a.$

We now compute the flow value. Since $U_a$ is empty for $a\in A_y^-$, using \eqref{empty} and \eqref{cons3},
we get
\begin{equation}\label{flowvalue}
v(f_{\bm{\gamma}})=\sum_{a\in A_y^+}|U_a|-\sum_{a\in A_y^-}|U_a|=\sum_{a\in A_y^+}|U_a|=m.
\end{equation}
Now the expression $f_{\bm{\gamma}}=\sum_{j=1}^m \chi_{\gamma_j}$ is an immediate consequence of \eqref{ind-flow-def}  and \eqref{chi-path}.
Finally observe that the equality \eqref{fgammax} holds also for $x\in\{y,z\}$ because, by Definition \ref{flow-per-x} and by \eqref{flowvalue}, we have 
$f_{\bm{\gamma}}(x)=v(f_{\bm{\gamma}})=m=l_x({\bm{\gamma}})$. 
\end{proof}

The above proposition allows  to give an important definition.

\begin{definition}\label{flow-ass-path}
Let $N=(K_V,c)\in \mathscr{N}$, $y,z\in V$ be distinct and $\bm{\gamma}\in \mathcal{S}_{yz}^{N}.$ The flow $f_{\bm{\gamma}}$ defined in \eqref{ind-flow-def} is called the flow associated with $\bm{\gamma}.$ 
\end{definition}

Note that if $\bm{\gamma}, \bm{\gamma}' \in \mathcal{S}_{yz}^{N}$ are equivalent, then $f_{\bm{\gamma}}=f_{\bm{\gamma}'}.$

 \subsection{From flows to paths}\label{flows-vs-paths}

In this section we present the well-known Flow Decomposition Theorem in a form that is useful for our purposes and explore its fundamental consequences for our research.   We will make large use of generalized path functions  and cycle functions (Definition \ref{path-function}).

We start recalling, within our notation, the well-known concept of augmenting path.
A generalized path $\gamma$ from $y$ to $z$  in $K_V$ is called an {\it augmenting path} for $f$ in $N$ if $c(a)-f(a)\geq 1$ for all $a\in A(\gamma)^+$, and $f(a)\geq 1$ for all $a\in A(\gamma)^-$.
We denote by $AP^N_{yz}(f)$ the set of the augmenting paths from $y$ to $z$ for $f$ in $N$.
By the celebrated Ford and Fulkerson Theorem, $f\in \mathcal{M}(N,y,z)$ if and only if $AP^N_{yz}(f)=\varnothing$. The next proposition is  a straightforward but useful interpretation of the flow augmenting path algorithm within our notation.
\begin{proposition}\label{aug} Let $N=(K_V,c)\in \mathscr{N}$, $y,z\in V$ be distinct  and $f\in \mathcal{F}(N,y,z)\setminus \mathcal{M}(N,y,z).$
Then $AP^N_{yz}(f)\neq \varnothing$ and, for every $\sigma\in AP^N_{yz}(f)$, we have that $f+\chi_\sigma\in \mathcal{F}(N,y,z)$ and $v(f+\chi_\sigma)=v(f)+1$.
\end{proposition}

The following result is substantially a technical rephrase of the Flow Decomposition Theorem as it is presented in Ahuja et al. (1993, Chapter 3). 
\begin{theorem}
\label{decomposition}
Let $N=(K_V,c)\in \mathscr{N}$, $y,z\in V$ be distinct and $f\in \mathcal{F}(N,y,z)$ having value $m\in \mathbb{N}_0$. Then there exist a sequence $\bm{\gamma}=(\gamma_j)_{j\in[m]}$ of $m$ paths from $y$ to $z$ in $N$, $k\in \mathbb{N}_0$  and a sequence $\bm{w}=(w_j)_{j\in [k]}$ of $k$ cycles in $N$ such that
\begin{equation}\label{dec}
f=\sum_{j\in[m]}\chi_{\gamma_j}+\sum_{j\in [k]}\chi_{w_j}.
\end{equation}
A couple $(\bm{\gamma},\bm{w})$ satisfying \eqref{dec} is called a decomposition of $f$.
For every decomposition $(\bm{\gamma},\bm{w})$ of $f$, we have
$\bm{\gamma}\in \mathcal{S}_{yz}^{N,m}$.
\end{theorem}

\begin{proof} Except for the final statement, everything comes from Ahuja et al. (1993, Theorem 3.5). We need only to show that $(\gamma_j)_{j\in[m]}\in \mathcal{S}_{yz}^{N,m}$. Assume then, by contradiction, that there exists $a\in A$ such that
$|\{j\in [m]: a\in A(\gamma_{j})\}|>c(a)$. Then, by \eqref{dec} and recalling that the cycle functions are non-negative, we deduce
$$f(a)=\sum_{j\in[m]}\chi_{\gamma_j}(a)+\sum_{j\in[k]}\chi_{w_j}(a)=|\{j\in [m]: a\in A(\gamma_{j})\}|+\sum_{j\in[k]}\chi_{w_j}(a)>c(a),$$ a contradiction.
\end{proof}

\begin{example}\label{example2-text}{\rm 
As an illustration of Theorem \ref{decomposition}, consider the network $N$ and the flow $f$ from $y$ to $z$ in $N$ described in Figure \ref{example2}.
\begin{figure}[t]
\begin{center}
\begin{tikzpicture}
\begin{scope}
    \node[style={circle,thick,draw,fill=white}] (y) at (0,0) {$y$};
    \node[style={circle,thick,draw,fill=lightgray}] (a) at (2,1) {$v$};
    \node[style={circle,thick,draw,fill=lightgray}] (c) at (4,1) {$x$};
    \node[style={circle,thick,draw,fill=lightgray}] (b) at (3,-1){$u$};
    \node[style={circle,thick,draw,fill=white}] (z) at (6,0) {$z$} ;
\end{scope}
\begin{scope}[>=stealth, every edge/.style={thick,draw}]
    \path [->] (y) edge node[pos=0.5,anchor=south]  {$1,1$} (a);
    \path [->] (y) edge node[pos=0.5,anchor=north]  {$1,1$} (b);
    \path [->] (a) edge node[pos=0.5,anchor=south]  {$2,2$} (c);
    \path [->] (b) edge node[pos=0.5,anchor=east]  {$1,1$} (a);
		\path [->] (c) edge node[pos=0.5,anchor=west]  {$1,1$} (b);
		\path [->] (b) edge node[pos=0.5,anchor=north]  {$1,1$} (z);
		\path [->] (c) edge node[pos=0.5,anchor=south]  {$1,1$} (z);
    
\end{scope}
\end{tikzpicture}
\end{center}
\caption{The Flow Decomposition Theorem}
\label{example2}
\end{figure}
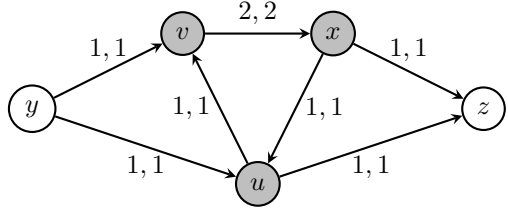
Note that the $v(f)=2$. A simple check shows that we have
\[
f=\chi_{yvxz}+\chi_{yuz}+\chi_{vxuv},
\]
where $(yvxz,yuz)$ is a sequence of 2 arc-disjoint paths from $y$ to $z$ in $N$ and $vxuv$ is a cycle in $N$.
Moreover we also have
\[
f=\chi_{yvxuz}+\chi_{yuvxz},
\]
where $(yvxuz,yuvxz)$ is a sequence of 2 arc-disjoint paths from $y$ to $z$ in $N$ and no cycle is involved.
In other words, $((yvxz,yuz),(vxuv))$ and $((yvxuz,yuvxz),())$ are two decompositions of $f$. That confirms the well-known fact that, in general, a flow can admit diverse decompositions. In particular, some involving cycles and some not.}
\end{example}

By Theorem \ref{decomposition}  we deduce that if there exists a flow of value $m$, then there also exists a flow of the same value of the type $f_{\bm{\gamma}}$, where $\bm{\gamma}\in \mathcal{S}_{yz}^{N,m}$. Such a $\bm{\gamma}$ can be obtained by considering any decomposition $(\bm{\gamma}^*,\bm{w}^*)$ of an arbitrarily chosen $m$-valued flow $f^*$ and setting $\bm{\gamma}=\bm{\gamma}^*$.

By Theorem \ref{decomposition} we can also better comment and comprehend the issue raised in Section \ref{intro-lambda}. Let us consider $\bm{\gamma}=(\gamma_{j})_{j\in[m]}\in \mathcal{S}^{N,m}_{yz}$, where $m\in \mathbb{N}$ and $m<n\leq  \varphi^{N}_{yz}$ (so that $\bm{\gamma}$ has not maximum length). Then $f_{\bm{\gamma}}$ is not a maximum flow. Applying the flow augmenting path algorithm $n-m$ times  using the augmenting paths $\sigma_1,\dots,\sigma_{n-m}$,
we find the maximum flow $\hat{f}=\sum_{j\in[m]}\chi_{\gamma_j}+\sum_{j\in [n-m]}\chi_{\sigma_j}$. Recall that the $\sigma_j$ are not paths but generalized paths. Now, by Theorem \ref{decomposition}, we have that there exist $\bm{\mu}=(\mu_{j})_{j\in[n]}\in \mathcal{S}^{N,n}_{yz}$ and a sequence $\bm{w}=(w_j)_{j\in [k]}$ of $k\in \mathbb{N}_0$ cycles in $N$ such that $\hat{f}=\sum_{j\in[n]}\chi_{\mu_j}+\sum_{j\in [k]}\chi_{w_j}$. The sequence $\bm{\mu}$ does not contain, in general, the original sequence $\bm{\gamma}$ as a subsequence and there is no immediate way to get one from the other.

Finally, Theorem \ref{decomposition} also allows to naturally associate with every flow a set of sequences of arc-disjoint paths in the sense of the following definition.

\begin{definition}\label{Sf-def}
 Let $N=(K_V,c)\in \mathscr{N}$, $y,z\in V$ be distinct and $f\in \mathcal{F}(N,y,z)$ with $v(f)=m$. We set
\begin{equation}\label{Sf}
\mathcal{S}^N_{yz}(f):=\big\{\bm{\gamma}\in \mathcal{S}^N_{yz}:\exists k\in\mathbb{N}_0\mbox{ and }\bm{w}\in (C^N)^k\mbox{ such that $(\bm{\gamma},\bm{w})$ is a decomposition of $f$}\big\},
\end{equation}
and we call $\mathcal{S}^N_{yz}(f)$ the set of sequences of arc-disjoint paths associated with $f.$
We also set
\begin{equation}\label{T}
\mathcal{T}^{N,m}_{yz}:=\bigcup_{\substack{{f\in \mathcal{F}(N,y,z)}\\{v(f)=m}}}\mathcal{S}^N_{yz}(f)\qquad\hbox{and}\qquad  \mathcal{T}^{N}_{yz}:=\mathcal{T}^{N,\varphi^N_{yz}}_{yz},
\end{equation}
and we call $\mathcal{T}^{N,m}_{yz}$ the set of sequences of  arc-disjoint paths for $m$-valued flows and $\mathcal{T}^{N}_{yz}$ the set of  sequences of arc-disjoint paths for maximum flows.
\end{definition}

Note that, by Theorem \ref{decomposition}, if $f\in \mathcal{F}(N,y,z)$ with $v(f)=m$, then $\varnothing\neq\mathcal{S}^N_{yz}(f)\subseteq\mathcal{S}_{yz}^{N,m}$.

We are now in position to clarify the link between sequences of arc-disjoint paths and flows in a network.
Proposition \ref{dec-x} below significantly extends \eqref{lex:disjointpaths} showing that, whatever is $m$, the sequences of $m$ arc-disjoint paths are exactly those associated with the flows of value $m$, through the Flow Decomposition Theorem. Moreover it shows that  the set of sequences of  arc-disjoint paths for maximum flows coincides with the set of maximum sequences of arc-disjoint paths.

\begin{proposition}\label{dec-x}
Let $N=(K_V,c)\in \mathscr{N}$, $y,z\in V$ be distinct  and $m\in\mathbb{N}_0$. Then the following facts hold:
\begin{itemize}
\item[$(i)$] $\mathcal{S}_{yz}^{N,m}=\mathcal{T}^{N,m}_{yz}$;
 \item[$(ii)$] $\mathcal{M}_{yz}^{N}=\mathcal{T}^{N}_{yz}$.
\end{itemize}
\end{proposition}

\begin{proof} $(i)$ Let $f\in \mathcal{F}(N,y,z)$ with $v(f)=m$. We have already observed that $\mathcal{S}^N_{yz}(f)\subseteq \mathcal{S}_{yz}^{N,m}.$ Thus, by \eqref{T}, we get
$\mathcal{T}^{N,m}_{yz}\subseteq \mathcal{S}_{yz}^{N,m}.$ Let now $\bm{\gamma}^*=(\gamma_{j}^*)_{j\in[m]}\in \mathcal{S}_{yz}^{N,m}$ and consider the flow  $f_{\bm{\gamma}^*}$ associated with $\bm{\gamma}^*$. By Proposition \ref{ind-flow}, we have that $v(f_{\bm{\gamma}^*})=m$ and $f_{\bm{\gamma}^*}=\sum_{j\in[m]} \chi_{\gamma_{j}^*}$, which means that 
we have a decomposition of $f_{\bm{\gamma}^*}$  given by $(\bm{\gamma}^*,())$ with no cycle involved.
Clearly, by \eqref{Sf} and \eqref{T}, we get $\bm{\gamma}^*\in \mathcal{S}^N_{yz}(f_{\bm{\gamma}^*})\subseteq \mathcal{T}^{N,m}_{yz}.$

$(ii)$ Apply $(i)$ to $m=\varphi^N_{yz}.$
\end{proof}

\section{En route for the proof of the main theorem}\label{dimo}

In this section we present some technical results to which we will appeal for the proof of the main theorem (Theorem \ref{main}).  To start with, given a maximum flow $f$, we show an interesting inequality between $f(x)$ and $\lambda^N_{yz}(x)$.

\begin{lemma}\label{lemma-nuovo}
Let $N=(K_V,c)\in \mathscr{N}$, $x,y,z\in V$ with $y,z$ distinct and $f\in \mathcal{F}(N,y,z)$ with $v(f)=m\in \mathbb{N}_0$. Then the following facts hold:
\begin{itemize}
\item[$(i)$] for every $\bm{\gamma}\in \mathcal{S}^N_{yz}(f)$, we have $f(x)\geq f_{\bm{\gamma}}(x)$;
\item[$(ii)$] if $f\in \mathcal{M}(N,y,z)$, then $f(x)\geq \lambda^N_{yz}(x)$;
\item[$(iii)$] if $\bm{\gamma}\in \mathcal{M}^{N}_{yz}(x)$, then $f_{\bm{\gamma}}(x)=\lambda_{yz}^N(x)$.

\end{itemize}
\end{lemma}

\begin{proof} $(i)$ If $x\in \{y,z\}$, then $f(x)=m=f_{\bm{\gamma}}(x).$ Assume next $x\notin \{y,z\}$.
Let $\bm{\gamma}\in \mathcal{S}^N_{yz}(f)$. By Definition  \ref{Sf-def}, there exists a  sequence $(w_j)_{j\in [k]}$ of cycles in $N$ such that
 $f=f_{\bm{\gamma}}+\sum_{j\in [k]}\chi_{w_j}$.
Thus, by Definition \ref{flow-per-x} and recalling that the cycle functions assume only non-negative value, we have
\[
f(x)=\sum_{a\in A_x^+} f(a)=\sum_{a\in A_x^+} f_{\bm{\gamma}}(a)+\sum_{a\in A_x^+} \left(\sum_{j\in [k]}\chi_{w_j}(a)\right)\geq f_{\bm{\gamma}}(x).
\]

$(ii)$ Assume that $f\in \mathcal{M}(N,y,z)$ and pick $\bm{\gamma}\in \mathcal{S}^N_{yz}(f)$.
By $(i)$ and by equality \eqref{fgammax}, we then have that $f(x)\geq f_{\bm{\gamma}}(x)=l_x(\bm{\gamma})$. Since, by Proposition \ref{dec-x}\,$(ii)$, we have $\mathcal{S}^N_{yz}(f)\subseteq \mathcal{M}^N_{yz}$ then we also have $$f(x)\geq \min_{\bm{\gamma}\in\mathcal{M}_{yz}^{N}}l_x(\bm{\gamma})= \lambda^N_{yz}(x).$$

$(iii)$ Let $\bm{\gamma}\in \mathcal{M}^{N}_{yz}(x)$. Then, by \eqref{fgammax}, we have that $f_{\bm{\gamma}}(x)=l_x(\bm{\gamma})=\lambda_{yz}^N(x)$.
\end{proof}

The next lemma establishes a natural bound for $\lambda^{N}_{yz}(x)$ in terms of the outdegree and the indegree of $x$.

\begin{lemma}\label{ind-outd}
Let $N=(K_V,c)\in\mathscr{N}$ and $x,y,z\in V$ be distinct. Then $\lambda^{N}_{yz}(x)\le \min\{c(x), c(V\setminus\{x\})\}.$
\end{lemma}

\begin{proof}
Consider $\bm{\gamma}  \in \mathcal{M}^{N}_{yz}(x)$. By Lemma \ref{lemma-nuovo}\,$(iii)$ and Definition \ref{flow-per-x}, we have
\[
\lambda^{N}_{yz}(x)=f_{\bm{\gamma}}(x)=\sum_{a\in  A_x^+}f_{\bm{\gamma}}(a)\le \sum_{a\in  A_x^+}c(a)=c(x)
\]
and also
\[
\lambda^{N}_{yz}(x)=f_{\bm{\gamma}}(x)=\sum_{a\in  A_x^-}f_{\bm{\gamma}}(a)\le \sum_{a\in  A_x^-} c(a)=c(V\setminus\{x\}).
\]
\end{proof}

In the following two results we explain how some crucial objects of our research behave with respect to a decrease of capacity in the network.
\begin{lemma}
\label{lex:lemmanegpaper}
Let $N =(K_V,c)\in\mathscr{N}$, $N'=(K_V,c')\in\mathscr{N}$ and $y,z\in V$ be distinct. Assume that $c'\le c$. Then, the following facts hold true:
\begin{itemize}
\item[$(i)$] $\mathcal{S}^{N'}_{yz}\subseteq\mathcal{S}^{N}_{yz}$. In particular, $\mathcal{M}^{N'}_{yz}\subseteq\mathcal{S}^{N}_{yz}$;
\item[$(ii)$] $\varphi^{N'}_{yz}(x)$ can be greater than $\varphi^{N}_{yz}(x).$
\end{itemize}
\end{lemma}
\begin{proof}
$(i)$ Let $\bm{\gamma}=(\gamma_j)_{j\in [m]}\in \mathcal{S}^{N'}_{yz}$, where $m\in\mathbb{N}_0$. Then, for every $a\in A$, we have
$$|\{j\in [m]: a\in A(\gamma_{j})\}|\le c'(a)\le c(a)$$
  and thus $\bm{\gamma}\in \mathcal{S}^{N}_{yz}$. Recall now that, by definition,  $\mathcal{M}^{N'}_{yz}\subseteq\mathcal{S}^{N'}_{yz}.$

$(ii)$ Consider the networks $N$ and $N'$ in Figures \ref{fig:N} and \ref{fig:Np} and denote by $c$ and $c'$ their capacities. 
Of course, we have that $c'\leq c$. It is easily checked that $\varphi^{N'}_{yz}(x)=1>\varphi^{N}_{yz}(x)=0$.
\begin{figure}[t]
\begin{minipage}{.5\textwidth}
\centering
\begin{tikzpicture}
\begin{scope}
    \node[style={circle,thick,draw,fill=white}]  (y) at (0,0) {$y$};
    \node[style={circle,thick,draw,fill=lightgray}]  (a) at (1,1) {$v$};
    \node[style={circle,thick,draw,fill=lightgray}]  (b) at (2,-1) {$u$};
    \node[style={circle,thick,draw,fill=lightgray}]  (c) at (3,1) {$x$};
    \node[style={circle,thick,draw,fill=white}]  (z) at (4,0) {$z$};
\end{scope}
\begin{scope}[>=stealth, every edge/.style={thick,draw}]
    \path [->] (y) edge node[pos=0.3,anchor=south] {1} (a);
    \path [->] (y) edge node[pos=0.5,anchor=north] {1} (b);
    \path [->] (a) edge node[pos=0.5,anchor=south] {1} (c);
    \path [->] (a) edge node[pos=0.5,anchor=north] {1} (z);
		\path [->] (c) edge node[pos=0.7,anchor=south] {1} (z);
		\path [->] (b) edge node[pos=0.5,anchor=north] {1} (z);
\end{scope}
\end{tikzpicture}					
\caption{The network $N$}
\label{fig:N}
\end{minipage}
\begin{minipage}{.5\textwidth}
\centering
\begin{tikzpicture}
\begin{scope}
    \node[style={circle,thick,draw,fill=white}] (y) at (0,0) {$y$};
    \node[style={circle,thick,draw,fill=lightgray}] (a)  at (1,1) {$v$};
    \node[style={circle,thick,draw,fill=lightgray}] (b)  at (2,-1) {$u$};
    \node[style={circle,thick,draw,fill=lightgray}] (c)  at (3,1) {$x$};
    \node[style={circle,thick,draw,fill=white}]  (z) at (4,0) {$z$};
\end{scope}
\begin{scope}[>=stealth, every edge/.style={thick,draw}]
    \path [->] (y) edge node[pos=0.3,anchor=south] {1} (a);
    \path [->] (y) edge node[pos=0.5,anchor=north] {1} (b);
    \path [->] (a) edge node[pos=0.5,anchor=south] {1} (c);
		\path [->] (c) edge node[pos=0.7,anchor=south] {1} (z);
		\path [->] (b) edge node[pos=0.5,anchor=north] {1} (z);
\end{scope}
\end{tikzpicture}					
\caption{The network $N'$}
\label{fig:Np}
\end{minipage}
\end{figure}
\end{proof}

\begin{lemma}
\label{lemma1-2}
Let $N =(K_V,c)\in\mathscr{N}$, $N'=(K_V,c')\in\mathscr{N}$ and $y,z\in V$ be distinct. Assume that $c'\le c$. 
Then the following conditions are equivalent:
\begin{itemize}
\item[$(a)$] $\mathcal{M}^{N'}_{yz}\subseteq\mathcal{M}^{N}_{yz}$;
\item[$(b)$] $\varphi^{N'}_{yz} = \varphi^{N}_{yz}$;
\item[$(c)$] $\mathcal{M}(N',y,z)\subseteq \mathcal{M}(N,y,z)$.
\end{itemize}
\end{lemma}
\begin{proof}
$(a)\Rightarrow (b)$ Assume that $\mathcal{M}^{N'}_{yz}\subseteq\mathcal{M}^{N}_{yz}$. Pick $\bm{\gamma}\in \mathcal{M}^{N'}_{yz}$.
Then $\varphi^{N'}_{yz}=l(\bm{\gamma})$ and also $\varphi^{N}_{yz}=l(\bm{\gamma})$, so that $\varphi^{N'}_{yz}=\varphi^{N}_{yz}.$

$(b)\Rightarrow (c)$ Assume that $\varphi^{N'}_{yz}=\varphi^{N}_{yz}.$ Let $f\in \mathcal{M}(N',y,z)$. Then, by \eqref{utile-inclusione},
$f\in \mathcal{F}(N,y,z)$ and  $v(f)=\varphi^{N'}_{yz}= \varphi^{N}_{yz}$. Thus, $f\in\mathcal{M}(N,y,z)$.

$(c)\Rightarrow (a)$ Assume that $\mathcal{M}(N',y,z)\subseteq \mathcal{M}(N,y,z)$. Let $\bm{\gamma}\in \mathcal{M}^{N'}_{yz}$. Then
$l(\bm{\gamma})=\varphi^{N'}_{yz}$ and, by Lemma \ref{lex:lemmanegpaper}$\,(i)$, $\bm{\gamma}\in \mathcal{S}^{N}_{yz}$. Consider the flow
$f_{\bm{\gamma}}$ associated with $\bm{\gamma}$, and recall that $v(f_{\bm{\gamma}})=l(\bm{\gamma})=\varphi^{N'}_{yz}$.
 Hence $f_{\bm{\gamma}}\in \mathcal{M}(N',y,z)$ and thus $f_{\bm{\gamma}}\in \mathcal{M}(N,y,z)$. Thus,
$l(\bm{\gamma})=\varphi^{N}_{yz}$, which gives $\bm{\gamma}\in \mathcal{M}^{N}_{yz}$.
\end{proof}

\section{Main theorem }\label{proofs}

We are finally ready to prove our main result.\footnote{Theorem \ref{main} was conjectured in Ghiggi (2018).}
\begin{theorem}\label{main}
Let $N =(K_V,c)\in\mathscr{N}$ and  $x,y,z\in V$  with $y,z$ distinct. Then $\lambda_{yz}^N(x)= \varphi^N_{yz}(x)$.
\end{theorem}
\begin{proof}
If $x\in \{y,z\}$, then we have $\lambda_{yz}^N(x)= \varphi^N_{yz}$ and $\varphi^{N_x}_{yz}=0$. Thus,  the equality $\lambda_{yz}^N(x)= \varphi^N_{yz}(x)$ is certainly true.
We complete the proof proving that, if $x\not\in \{y,z\}$, then we have
$\lambda_{yz}^N(x)= \varphi^N_{yz}(x)$.
Observe first  that
\begin{equation}\label{s0}
\mbox{$\lambda^{N}_{yz}(x)=0\;\;$ implies  $\;\;\lambda^{N}_{yz}(x)=\varphi^{N}_{yz}(x)$.}
\end{equation}
Indeed, if $\lambda^{N}_{yz}(x)=0$, then there exists $\bm{\gamma}\in \mathcal{M}^{N}_{yz}$ such that $l_x(\bm{\gamma})=0$.
Thus, $\bm{\gamma}\in \mathcal{S}^{N_x}_{yz}$, which gives $\varphi^{N_{x}}_{yz}\geq \varphi^{N}_{yz}$. Since by \eqref{utile-disu} we have $\varphi^{N_{x}}_{yz}=\varphi^{N}_{yz}$, we deduce that  $\lambda^{N}_{yz}(x)=\varphi^{N}_{yz}(x)=0$.

Consider now, for $n\in \mathbb{N}_0$, the following statement:
\begin{equation}\label{statement-n}
\begin{array}{l}
\mbox{For every $N=(K_V,c)\in\mathscr{N}$, $x,y,z\in V$ distinct  and $c(x) + c(V\setminus \{x\})=n$,}\\
\mbox{we have that  $\lambda^{N}_{yz}(x)=\varphi^{N}_{yz}-\varphi^{N_{x}}_{yz}$.}
\end{array}
\end{equation}
We are going  to prove the theorem showing, by induction on $n$, that \eqref{statement-n} holds true for all $n\in \mathbb{N}_0$.

Consider first $N =(K_V,c)\in\mathscr{N}$ and $x,y,z\in V$ distinct with $c(x) + c(V\setminus \{x\})=0.$
Then $c(x)=c(V\setminus \{x\})=0$ which,  by Lemma \ref{ind-outd}, implies $\lambda^{N}_{yz}(x) = 0$ and,  by \eqref{s0}, the statement holds.

Consider now $N =(K_V,c)\in\mathscr{N}$,  and $x,y,z\in V$ distinct with $c(x) + c(V\setminus \{x\})=n\ge 1$.
For brevity, let us set $\lambda^{N}_{yz}(x)=s$ and $\varphi^{N}_{yz}=m$.
By Proposition \ref{ineq-utile}$(i)$, we have that $0\leq s\le m$.
If $s=0$, then we again conclude by \eqref{s0}. Assume then $s\ge 1$. As a consequence, we also have $m\ge 1$.  Choose among the sequences in $ \mathcal{M}^{N}_{yz}(x)$  a sequence $\bm{\gamma}  \in \mathcal{M}^{N}_{yz}(x)$ in which the components passing through $x$ are
 the last  $s$.
 Let  $f_{\bm{\gamma}}$  be the flow associated with $\bm{\gamma}$,
defined in \eqref{ind-flow-def}. Recall that $l(\bm{\gamma})=m$ and $l_x(\bm{\gamma})=s$.

\vspace{2mm}

We divide our argument into two cases.

\vspace{2mm}

\noindent {\it Case (I).} Assume that there exists $\tilde{a}\in  A_{x}$ such that $f_{\bm{\gamma}}(\tilde{a}) < c(\tilde{a}).$
Then, obviously, $c(\tilde{a})\ge 1$.

Consider the network $\tilde{N}=(K_V, \tilde{c})$ where $\tilde{c}$ is defined, for every $a\in A$, as
\begin{equation*}
\tilde{c}(a) = \begin{cases}c(a)\quad\quad\quad \hbox{if} \quad a \ne \tilde{a} \\ c(\tilde{a}) - 1 \quad\,\, \hbox{if}\quad  a = \tilde{a} \end{cases}
\end{equation*}
and note now that, for every $a\in A$, $\tilde{c}(a)\le c(a)$.
Since $\tilde{c}(x) + \tilde{c}(V\setminus \{x\})=c(x) + c(V\setminus \{x\})-1=n-1$, by inductive assumption we get
\begin{equation*}
\label{eq:thesisfortilde}
\lambda^{\tilde{N}}_{yz}(x) = \varphi^{\tilde{N}}_{yz} - \varphi^{\tilde{N}_{x}}_{yz}.
\end{equation*}

It is immediate to observe that $\tilde{N}_{x} = N_{x}$, so that $\varphi^{\tilde{N}_{x}}_{yz} = \varphi^{N_{x}}_{yz}$.
We also have $\bm{\gamma}\in\mathcal{S}^{\tilde{N}}_{yz}$ and then
$\varphi^{\tilde{N}}_{yz}\ge l(\bm{\gamma})=m= \varphi^{N}_{yz}$. Moreover, by \eqref{utile-disu}, we also have $\varphi^{\tilde{N}}_{yz}\le \varphi^{N}_{yz}$. Thus, $\varphi^{\tilde{N}}_{yz} = \varphi^{N}_{yz}$ and $\bm{\gamma}\in\mathcal{M}^{\tilde{N}}_{yz}$. As a consequence, $\varphi^{N}_{yz} - \varphi^{N_x}_{yz}=\varphi^{\tilde{N}}_{yz} - \varphi^{\tilde{N}_{x}}_{yz}$.
We are then left with proving that  $\lambda^{\tilde{N}}_{yz}(x) =\lambda^{N}_{yz}(x)$.
Note  that $\lambda^{\tilde{N}}_{yz}(x)\le l_x(\bm{\gamma})=s.$
Assume now, by contradiction, that there exists  $\tilde{\bm{\gamma}}\in\mathcal{M}^{\tilde{N}}_{yz}$ such that  $l_x(\tilde{\bm{\gamma}})<s$.
As proved before, $\varphi^{\tilde{N}}_{yz} = \varphi^{N}_{yz}$ and then, by Lemma \ref{lemma1-2}, we have that $\tilde{\bm{\gamma}}\in\mathcal{M}^{N}_{yz}$ and then
 $\lambda^{N}_{yz}(x)\le l_x(\tilde{\bm{\gamma}})<s$, a contradiction.

\vspace{2mm}

\noindent {\it Case (II).}  Assume now that, for every $a\in  A_{x}$, we have
\begin{equation}
\label{lex:point2paper}
f_{\bm{\gamma}}(a)= c(a).
\end{equation}
By Lemma \ref{lemma-nuovo}$(iii)$, we then get
\begin{equation}\label{star}
s=\sum_{a\in  A_x^+}f_{\bm{\gamma}}(a)=\sum_{a\in  A_x^+}c(a)=c(x).
\end{equation}

The component $\gamma_m$ of $\bm{\gamma}$ passes through $x$ and reaches $z\neq x$. Thus there exists $\tilde{a} \in A_{x}^{+}\cap A(\gamma_m)$  and, by \eqref{lex:point2paper}, we have
$c(\tilde{a})\geq 1.$

Define the network $\tilde{N} = (K_V, \tilde{c}) \in \mathscr{N}$ by:
\begin{equation*}
\tilde{c}(a) = \begin{cases}c(a)\quad\quad\hbox{ if }\,\,a \ne \tilde{a} \\ c(\tilde{a}) - 1 \ \hbox{ if}\ a = \tilde{a}. \end{cases}
\end{equation*}
and note now that, for every $a\in A$, $\tilde{c}(a)\le c(a)$.
Since $\tilde{c}(x)=c(x)-1$ and $\tilde{c}(V\setminus \{x\})=c(V\setminus \{x\})$, we have that $\tilde{c}(x) + \tilde{c}(V\setminus \{x\})=n-1$. Hence, by inductive assumption, we get
$
\lambda^{\tilde{N}}_{yz}(x) = \varphi^{\tilde{N}}_{yz} - \varphi^{\tilde{N}_{x}}_{yz}.
$
In order to complete the proof we show the following three equalities:
\begin{itemize}
\item [$(a)$] $\varphi^{\tilde{N}}_{yz} = \varphi^{N}_{yz} - 1$;
\item [$(b)$] $\lambda^{\tilde{N}}_{yz}(x) = \lambda^{N}_{yz}(x) - 1$;
\item [$(c)$] $\varphi^{\tilde{N}_{x}}_{yz} = \varphi^{N_{x}}_{yz}.$
\end{itemize}

Let us start by considering  $\tilde{\bm{\gamma}}\in \mathcal{S}^{N}_{yz} $ obtained by $\bm{\gamma}$ by deleting the component $\gamma_m$. In other words,
$\tilde{\bm{\gamma}}=(\tilde{\gamma}_j)_{j\in [m-1]}$ where, for every $j\in [m-1]$, $\tilde{\gamma}_j=\gamma_j$.
By definition of $\tilde{N}$, we surely have $\tilde{\bm{\gamma}}\in \mathcal{S}^{\tilde{N}}_{yz}$ and thus
\begin{equation}\label{prima}
\varphi^{\tilde{N}}_{yz}\ge l(\tilde{\bm{\gamma}}) = m - 1 = \varphi^{N}_{yz} - 1.
\end{equation}
Moreover, by Lemma \ref{ind-outd} and \eqref{star}, we have
\begin{equation}
\label{lex:appendixpaper}
\lambda^{\tilde{N}}_{yz}(x) \le \tilde{c}(x) = c(x) -1=s-1.
\end{equation}
Let us now prove the equalities $(a)$, $(b)$ and $(c)$.

$(a)$ Assume by contradiction that $\varphi^{\tilde{N}}_{yz} > \varphi^{N}_{yz} - 1$, that is, $\varphi^{\tilde{N}}_{yz} \ge \varphi^{N}_{yz}$.
By \eqref{utile-disu}, we then obtain $\varphi^{N}_{yz} = \varphi^{\tilde{N}}_{yz}$. By Lemma \ref{lemma1-2}, we also deduce that
$\mathcal{M}^{\tilde{N}}_{yz} \subseteq \mathcal{M}^{N}_{yz}$ so that $\lambda^{\tilde{N}}_{yz}(x) \ge s.$
On the other hand, by \eqref{lex:appendixpaper}, we also have $\lambda^{\tilde{N}}_{yz}(x) \le s - 1$,  a contradiction. As a consequence, $\varphi^{\tilde{N}}_{yz} \le  \varphi^{N}_{yz} - 1$. Using now \eqref{prima}, we conclude $\varphi^{\tilde{N}}_{yz} =  \varphi^{N}_{yz} - 1$, as desired.

$(b)$ Let us prove now $\lambda^{\tilde{N}}_{yz}(x) = s - 1$. By \eqref{lex:appendixpaper} it is enough to show
\begin{equation}
\label{lex:pointbsez2paper}
\lambda^{\tilde{N}}_{yz}(x) \ge s - 1.
\end{equation}
Set $\lambda^{\tilde{N}}_{yz}(x)=\tilde{s}$. From ($a$) we know that $\varphi^{\tilde{N}}_{yz}=m-1$.
Let $\tilde{\bm{\nu}}\in\mathcal{M}^{\tilde{N}}_{yz}(x)$.
Thus $l(\tilde{\bm{\nu}})=m-1$ and $l_x(\tilde{\bm{\nu}})=\tilde{s}.$
By Proposition \ref{ind-flow}, we have that $f_{\tilde{\bm{\nu}}}\in \mathcal{F}(\tilde{N},y,z)\subseteq \mathcal{F}(N,y,z)$ and $v(f_{\tilde{\bm{\nu}}})=m-1.$
Thus $f_{\tilde{\bm{\nu}}}\in\mathcal{F}(N,y,z)\setminus \mathcal{M}(N,y,z)$ so that
 $AP^N_{yz}(f_{\tilde{\bm{\nu}}})\neq \varnothing$.
Pick then $\sigma \in AP^N_{yz}(f_{\tilde{\bm{\nu}}})$. By Proposition \ref{aug}, we have that $f=f_{\tilde{\bm{\nu}}}+\chi_{\sigma}\in \mathcal{M}(N,y,z).$
By \eqref{fgammax}, we then have
\begin{equation}\label{calcolo1}
2f(x)=\sum_{a\in A^+_x}f(a)+\hspace{-1mm}\sum_{a\in A^-_x}f(a)=\hspace{-1mm}\sum_{a\in A_x}f(a)=\hspace{-1mm}\sum_{a\in A_x}f_{\tilde{\bm{\nu}}}(a)+\hspace{-1mm}\sum_{a\in A_x}\chi_{\sigma}(a)=2f_{\tilde{\bm{\nu}}}(x)+\hspace{-1mm}
\sum_{a\in A_x}\chi_{\sigma}(a) \leq 2f_{\tilde{\bm{\nu}}}(x)+2.
\end{equation}
The last inequality follows from the fact that
\begin{equation}\label{one}
\sum_{a\in A_x}\chi_{\sigma}(a)\le 2.
\end{equation}
Indeed, by definition \eqref{chi-generalizzato}, we have
\[
\chi_{\sigma}=\sum_{a\in A(\sigma)^+} \chi_{a}- \sum_{a\in A(\sigma)^-} \chi_{a}.
\]
In particular, $\chi_{\sigma}(a)=0$ for all $a\in A_x\setminus A(\sigma)$ and $\chi_{\sigma}(a)\leq 1$ for all $a\in A_x\cap A(\sigma).$
Now, by definition of generalized path,
we have  $|A_x\cap A(\sigma)|\in\{0, 2\}$ and thus \eqref{one} holds.

By \eqref{calcolo1} and \eqref{fgammax}, we then obtain $f(x)\leq f_{\tilde{\bm{\nu}}}(x)+1=\tilde{s}+1$.
On the other hand, by Lemma \ref{lemma-nuovo}\,$(ii)$, we  also have
$
s\leq f(x)
$
and thus $s\le \tilde{s}+1$, which is \eqref{lex:pointbsez2paper}.

$(c)$ Clearly we have that  $N_{x} = \tilde{N}_{x}$ and thus
$\varphi^{\tilde{N}_{x}}_{yz} = \varphi^{N_{x}}_{yz}.$
\end{proof}

The next proposition shows that the equality $\lambda_{yz}^N(X)=\varphi_{yz}^N(X)$ does not hold true in general when $X$ is not a singleton.
Proposition \ref{bang} follows by an example due to Bang-Jensen (2019).

\begin{proposition}\label{bang}	
There exist $N=(K_V,c)\in\mathscr{N}, X\subseteq V$ and $y,z\in V$ distinct such that  $\lambda_{yz}^N(X)> \varphi_{yz}^N(X).$
\end{proposition}

\begin{proof} Consider the network $N$ in Figure \ref{fig:BJ} and $X=\{x_1,x_2\}$. Then we have 
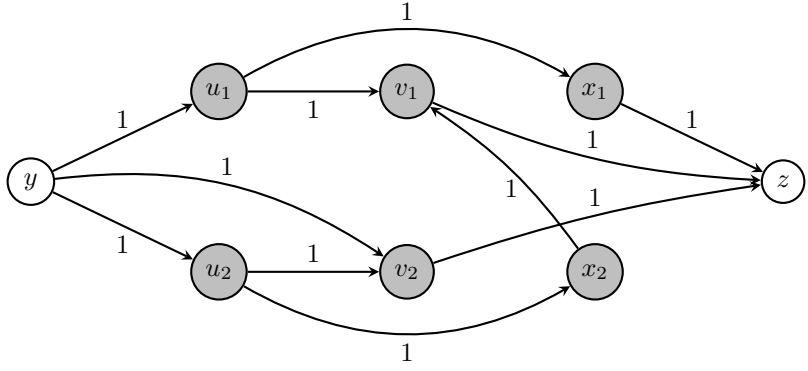
\begin{figure}[t]
\begin{center}
\begin{tikzpicture}
\begin{scope}[every node/.style={circle,thick,draw,fill=lightgray}]
    \node[style={circle,thick,draw,fill=white}]  (A) at (0,0) {$y$};
    \node[style={circle,thick,draw,fill=lightgray}]  (B) at (2.5,1.2) {$u_1$};
    \node[style={circle,thick,draw,fill=lightgray}]  (C) at (2.5,-1.2) {$u_2$};
    \node[style={circle,thick,draw,fill=lightgray}]  (D) at (5,1.2) {$v_1$};
    \node[style={circle,thick,draw,fill=lightgray}]  (E) at (5,-1.2) {$v_2$};
    \node[style={circle,thick,draw,fill=lightgray}]  (F) at (7.5,1.2) {$x_1$} ;
		\node[style={circle,thick,draw,fill=lightgray}]  (G) at (7.5,-1.2) {$x_2$} ;
		\node[style={circle,thick,draw,fill=white}]  (H) at (10,0) {$z$} ;
\end{scope}
\begin{scope}[>=stealth, every edge/.style={thick,draw}]
    \path [->] (A) edge node[pos=0.5,anchor=south] {1} (B);
    \path [->] (A) edge node[pos=0.5,anchor=north] {1} (C);
    \path [->] (B) edge node[pos=0.5,anchor=north] {1} (D);
    \path [->] (C) edge node[pos=0.5,anchor=south] {1} (E);
    \path [->] (D) edge[bend right=10] node[pos=0.5,anchor=south] {1} (H);
    \path [->] (E) edge[bend left=5] node[pos=0.5,anchor=south] {1} (H);
    \path [->] (A) edge[bend left=20] node[pos=0.5,anchor=south] {1} (E);
    \path [->] (C) edge[bend right=30] node[pos=0.5,anchor=north] {1} (G);
    \path [->] (G) edge[bend right=10] node[pos=0.5,anchor=north] {1} (D);
		\path [->] (B) edge[bend left=30] node[pos=0.5,anchor=south] {1} (F);
		\path [->] (F) edge node[pos=0.5,anchor=south] {1} (H);
\end{scope}
\end{tikzpicture}
\end{center}
\caption{$\lambda_{yz}^N(X)> \varphi_{yz}^N(X)$ for $X=\{x_1,x_2\}$}\label{fig:BJ}
\end{figure}
 It is immediately checked that $\varphi^N_{yz}=3$ and $\varphi^{N_X}_{yz}=2$, so that $\varphi^N_{yz}(X)=1$.
We show that $\lambda_{yz}^N(X)>\varphi^N_{yz}(X)$ proving that $\lambda_{yz}^N(X)=2$. Consider
\[
\bm{\gamma}=(yv_2z,\, yu_2x_2v_1z,\, yu_1x_1z)\in\mathcal{M}_{yz}^{N}.
\]
Since $l_X(\bm{\gamma})=2$, we have that $\lambda_{yz}^N(X)\leq 2$. Moreover, by \eqref{ineq-utile1}, we know that $\lambda_{yz}^N(X)\geq \varphi^N_{yz}(X)=1.$ Thus $\lambda_{yz}^N(X)\in \{1,2\}.$

Assume, by contradiction, that $\lambda_{yz}^N(X)=1.$ Then there exists $\bm{\mu}=(\mu_1,\mu_2,\mu_3)\in\mathcal{M}_{yz}^{N}(X)$ such that only $\mu_3$ passes through $X.$ Thus, $\tilde{\bm{\mu}}=(\mu_1,\mu_2)\in\mathcal{S}_{yz}^{N_X}$ and, since $\varphi^{N_X}_{yz}=2$, we deduce that $\tilde{\bm{\mu}}\in\mathcal{M}_{yz}^{N_X}.$ Hence, it is immediately observed that there are only two possibilities for a sequence of 2 arc-disjoint paths from $y$ to $z$ in $N_X$ (up to reordering of the components). More precisely, we have
\[
\tilde{\bm{\mu}}=(yu_1v_1z,yu_2v_2z)\quad \mbox{ or }\quad \tilde{\bm{\mu}}=(yu_1v_1z,yv_2z).
\]

If $\tilde{\bm{\mu}}=(yu_1v_1z,yu_2v_2z)$, then $\bm{\mu}=(yu_1v_1z,yu_2v_2z,\mu_3)$. Assume first that $\mu_3$ passes through $x_1$. Then the only arc entering into $x_1$ and having capacity $1$, that is $(u_1,x_1)$, must be an arc of $\mu_3$. That forces $A(\mu_3)$ to contain also the arc $(y, u_1)$. On the other hand, that arc is also an arc of $yu_1v_1z$ and we contradict the independence requirement.
Similarly, if $\mu_3$ passes through $x_2$, then the only arc entering into $x_2$ and having capacity $1$, that is $(u_2,x_2)$, must be an arc of $\mu_3$. That forces $A(\mu_3)$ to contain also the arc $(y, u_2)$, which is an arc of the path $yu_2v_2z$, again
against the independence requirement.

If now $\tilde{\bm{\mu}}=(yu_1v_1z,yv_2z)$, then $\bm{\mu}=(yu_1v_1z,yv_2z,\mu_3)$.
As in the previous case, there is no way to include $x_1$ as a vertex of $\mu_3$. Moreover, if $\mu_3$ passes through $x_2$, then necessarily $\mu_3=yu_2x_2v_1z$. Hence, $A(\mu_3)$ must contain the arc $(v_1,z)$, which is an arc of $yu_1v_1z$ against the independence requirement.
\end{proof}

Proposition \ref{bang} ultimately clarifies that the numbers $\lambda_{yz}^N(X)$ and $\varphi_{y,z}^N(X)$ stem from different ideas and that any alleged intuition above their equality is wrong. In other words, Theorem \ref{main} is a pure miracle.

\section{Further properties of $\varphi_{y,z}^N(X)$ and $\lambda_{yz}^N(X)$}\label{further}

Let us introduce a new concept based on Definition \ref{flow-per-x}.

\begin{definition}\label{delta}
Let  $N=(K_V,c)\in \mathscr{N}$, $y,z\in V$ be distinct and $X\subseteq V$. We define
\[
\delta^N_{yz}(X):=\min_{f\in \mathcal{M}(N,y,z)}f(X).
\]
\end{definition}
The number $\delta^N_{yz}(X)$ represents the global flow that must pass through $X$ in any maximum flow.
Since $\mathcal{M}(N,y,z)\ne\varnothing$, $\delta^{N}_{yz}(X)$ is well defined. 
Note that $X\cap \{y,z\}\neq \varnothing$ implies $\delta^N_{yz}(X)\geq \varphi^N_{yz}$.\footnote {It is easily seen that 
\[
\delta^N_{yz}(V)=\varphi^N_{yz}+\sum_{a\in A\setminus (A_y^-\cup A_z^+)}f(a)
\]}
Moreover, given $X\subseteq Y\subseteq V$, it is immediately observed that $0\le \delta^N_{yz}(X)\le \delta^N_{yz}(Y)$.

Proposition \ref{caratt-lambda}, Corollary \ref{corollary} and Proposition \ref{ineq-utile} show some interesting links among $\varphi_{yz}^N(X)$, $\lambda_{yz}^N(X)$  and $\delta_{yz}^N(X)$. In particular, we get that $\delta_{yz}^N(X)$ provides a characterization of $\varphi_{yz}^N(X)$ and $\lambda_{yz}^N(X)$ when $X$ is a singleton.

\begin{proposition}\label{caratt-lambda}
Let $N =(K_V,c)\in\mathscr{N}$, $y,z\in V$  be distinct and $X\subseteq V$. Then $\delta^N_{yz}(X)\geq \lambda^N_{yz}(X)$ and equality holds when $X$ is a singleton.
\end{proposition}
\begin{proof} Assume first that $X\cap\{y,z\}\ne \varnothing.$  Then we have $\lambda^N_{yz}(X)=\varphi^N_{yz}\leq \delta^N_{yz}(X).$
If $X=\{x\},$  so that $x=y$ or $x=z$, recalling Definition \ref{flow-per-x}, we have  that for every  $f\in \mathcal{M}(N,y,z)$, $f(x)= \varphi^N_{yz}$ and so  also $ \delta^N_{yz}(x)= \varphi^N_{yz}=\lambda^N_{yz}(x).$ 

Assume next that $X\cap\{y,z\}= \varnothing.$  Let $f\in \mathcal{M}(N,y,z)$ and $(\bm{\gamma},\bm{w})$ be a decomposition of $f$. By Lemma \ref{lemma-nuovo}\,$(i)$, we have 
\begin{equation}\label{primo}
f(X)=\sum_{x\in X}f(x)\geq \sum_{x\in X}f_{\bm{\gamma}}(x)=f_{\bm{\gamma}}(X),
\end{equation}
where, by Proposition \ref{ind-flow}, $f_{\bm{\gamma}}\in \mathcal{M}(N,y,z)$ and $\bm{\gamma}\in \mathcal{M}^N_{yz}$. We now observe that, by \eqref{fgammax}, we have
\begin{equation}\label{secondo}
f_{\bm{\gamma}}(X)=\sum_{x\in X}f_{\bm{\gamma}}(x)=\sum_{x\in X}l_x({\bm{\gamma}})\geq l_X({\bm{\gamma}})
\end{equation}
where the inequality in the chain is an equality when $X$ is a singleton.
By \eqref{primo} and \eqref{secondo} it then follows 
\begin{equation*}\label{terzo}
\delta^N_{yz}(X)=\min_{f\in \mathcal{M}(N,y,z)} f(X)=\min_{\bm{\gamma}\in \mathcal{M}^N_{yz}} f_{\bm{\gamma}}(X)\geq \min_{\bm{\gamma}\in \mathcal{M}^N_{yz}}l_{X}(\bm{\gamma})=\lambda^{N}_{yz}(X)
\end{equation*}
and  the inequality in the chain is an equality when $X$ is a singleton.
\end{proof}

\begin{corollary}\label{corollary}
Let $N =(K_V,c)\in\mathscr{N}$ and  $x,y,z\in V$  with $y,z$ distinct. Then $\lambda_{yz}^N(x)= \varphi^N_{yz}(x)=\delta_{yz}^N(x)$.
\end{corollary}
\begin{proof}
Simply apply Theorem \ref{main} and Proposition \ref{caratt-lambda}.
\end{proof}

\begin{proposition}\label{ineq-utile}	
Let $N=(K_V,c)\in\mathscr{N}$, $y,z\in V$ be distinct and  $X\subseteq V$. Then 
\begin{equation*}\label{ineq-utile1}
0\le \varphi_{yz}^N(X)\leq \lambda_{yz}^N(X)\leq \min\{\delta_{yz}^N(X),\varphi_{yz}^N\}.
\end{equation*}
\end{proposition}

\begin{proof} 
Let us prove first that $\varphi_{yz}^N(X)\ge 0$. Consider the network $N_X$ and observe that, for every $a\in A$,  $c_X(a)\le c(a)$. 
Thus, by \eqref{utile-disu}, $\varphi_{yz}^N\ge\varphi^{N_X}_{yz}$. We deduce then that $\varphi_{yz}^N(X)=\varphi_{yz}^N-\varphi^{N_X}_{yz}\ge 0$, as desired.

Let us prove now that $\varphi_{yz}^N(X)\leq \lambda_{yz}^N(X)\leq \varphi_{yz}^N$.
Let $\bm{\gamma}\in\mathcal{M}^{N}_{yz}(X).$ Then $l(\bm{\gamma})=\varphi^N_{yz}$ and $l_X(\bm{\gamma})=\lambda_{yz}^N(X)$.
Let $\bm{\gamma}'$ be a sequence of arc-disjoint paths having as components those components of $\bm{\gamma}$ not passing through $X$. Thus, 
$\bm{\gamma}'\in\mathcal{S}^{N_X}_{yz}$ and $l(\bm{\gamma}')=\varphi^N_{yz}-\lambda_{yz}^N(X) \le  \lambda^{N_X}_{yz} =\varphi^{N_X}_{yz}$, where the last equality follows from \eqref{lex:disjointpaths}. As a consequence,
$\varphi_{yz}^N(X)= \varphi^N_{yz}-\varphi^{N_X}_{yz}\le \lambda_{yz}^N(X)$.
Next note that $\lambda^{N}_{yz}(X)\le l_X(\bm{\gamma}) \le l(\bm{\gamma}) = \varphi^{N}_{yz}$. 

Finally the fact that $\lambda_{yz}^N(X)\leq \delta_{yz}^N(X)$ follows from Proposition \ref{caratt-lambda}.
\end{proof}

By Proposition \ref{bang} we already know that it may happen that  $\lambda_{yz}^N(X)> \varphi_{yz}^N(X).$
The next proposition shows that also the equality $\delta_{yz}^N(X)=\lambda_{yz}^N(X)$ does not hold true in general when $X$ is not a singleton.

\begin{proposition}\label{esempio-cammino}	
There exist $N=(K_V,c)\in\mathscr{N}, X\subseteq V$ and $y,z\in V$ distinct such that  $\delta_{yz}^N(X)> \lambda_{yz}^N(X).$
\end{proposition}

\begin{proof} Consider the network $N$ in Figure \ref{fig:cammino} and $X=\{x_1,x_2\}$. 
\begin{figure}[t]
\begin{center}
\begin{tikzpicture}
\begin{scope}
    \node[style={circle,thick,draw,fill=white}] (y) at (0,0) {$y$};
    \node[style={circle,thick,draw,fill=lightgray}] (a) at (2,0) {$x_1$};
    \node[style={circle,thick,draw,fill=lightgray}] (b) at (4,0) {$u$};
    \node[style={circle,thick,draw,fill=lightgray}] (c) at (6,0){$x_2$};
    \node[style={circle,thick,draw,fill=white}] (z) at (8,0) {$z$} ;
\end{scope}
\begin{scope}[>=stealth, every edge/.style={thick,draw}]
   \path [->] (y) edge node[pos=0.5,anchor=south]  {$1$} (a);
    \path [->] (a) edge node[pos=0.5,anchor=south]  {$1$} (b);
    \path [->] (b) edge node[pos=0.5,anchor=south]  {$1$} (c);
   \path [->] (c) edge node[pos=0.5,anchor=south]  {$1$} (z);
		
\end{scope}
\end{tikzpicture}
\end{center}
\caption{}\label{fig:cammino}
\end{figure}
Then  $\delta_{yz}^N(X)=2>\lambda_{yz}^N(X)=1.$
\end{proof}

\section{Two flow group centrality measures}\label{threecm}

Consider the ordered pairs of the type $((K_V,c),X)$, where 
$(K_V,c)\in \mathscr{N}$ and $X\subseteq V$ and denote the set of such pairs by $\mathscr{U}$.
A {\it group centrality measure} ({\sc gcm}) is a function from $\mathscr{U}$ to $\mathbb{R}$.  If $\mu$ is a group centrality measure, we denote the value of $\mu$ at $(N,X)\in\mathscr{U}$ by $\mu^N(X)$ and we interpret it as a measure of the importance of the set of vertices $X$ in $N$. 
A variety of group centrality measures have been obtained by properly generalizing classic centrality measures (Everett and Borgatti 1999, 2005).

By means of the numbers $\varphi_{yz}^N(X)$ and $\lambda_{y,z}^N(X)$, we define in this section two new group centrality measures. 
Recall that, given $N=(K_V,c)\in\mathscr{N}$, $y,z\in V$ distinct and $X\subseteq V$, we have that
$0\le \varphi_{yz}^N(X)\le \varphi_{yz}^N(V)$ and $0\le \lambda_{yz}^N(X)\le \lambda_{yz}^N(V)$. Moreover,
$\varphi_{yz}^N(V)=\lambda_{yz}^N(V)=\varphi_{yz}^N=\lambda_{yz}^N$.

\begin{definition}
The {\it full flow vitality }{\sc gcm},
denoted by $\varphi$ is defined, for every $(N,X) \in \mathscr{U}$ with $N=(V,A,c)$, by
\begin{equation}\label{phi-cent}
\varphi^N(X):=\sum_{\substack{{(y,z)\in A}\\{\varphi^N_{yz}(V)>0}}} \frac{\varphi^N_{yz}(X)}{\varphi^N_{yz}(V)} .
\end{equation}
\end{definition}
Note that $\varphi$ is a vitality measure in the sense of Kosch\"utzki et al. (2005). Indeed its value at a given set of vertices $X$ takes into consideration how much eliminating the set $X$ from the network impacts the global flow of the network.

We emphasize that in \eqref{phi-cent} we are summing over the maximum set of arcs which makes the definition  meaningful, that is over the set
$\{(y,z)\in A: \varphi^N_{yz}>0\}$. That corresponds to the idea of a uniform treatment for the vertices in the network and is what the adjective {\it full} in the name of $\varphi$ refers to.\footnote{Kosch\"utzki et al. (2005) call max-flow betweenness vitality of the vertex $x$ of an undirected connected network $N$, the number $\displaystyle{\sum_{\substack{{(y,z)\in A}\\{y\neq x\neq z}}}\frac{\varphi^{N}_{yz}(x)}{\varphi^{N}_{yz}}}$, where the sum is implicitly taken in a subset of $\{(y,z)\in A: \varphi^N_{yz}>0\}$. See also Borgatti and Everett (2006, p.475). }
Other choices are possible for applications in which a differentiation of vertices is reasonable.

\begin{definition}
The  {\it full flow betweenness} {\sc gcm},
denoted by $\lambda$, is defined, for every $(N,X) \in \mathscr{U}$ with $N=(V,A,c)$, by
\begin{equation}\label{lambda-cent}
\lambda^N(X):=\sum_{\substack{{(y,z)\in A}\\{\lambda^N_{yz}(V)>0}}}\frac{\lambda^N_{yz}(X)}{\lambda^N_{yz}(V)}.
\end{equation}
\end{definition}

Note that  $\lambda$ is a typical betweenness measure because it takes into considerations to which extent the paths in the network are forced to pass through a set of vertices. In line with \eqref{phi-cent}, also in \eqref{lambda-cent} we are summing over the maximum set of arcs which makes the definition  meaningful, that is over the set $\{(y,z)\in A: \lambda^N_{yz}>0\}$ and that explains the adjective {\it full} in the name of $\lambda$.

Being a betweenness measure makes $\lambda$ conceptually different from $\varphi$, which is instead a typical vitality measure. Anyway some comparison is surely possible.
First of all it is immediate to check that $\lambda$ and $\varphi$ coincides when $X=\{x\}$. Indeed by Theorem \ref{main} we know that, for every $y,z\in V$ distinct, we have  $\lambda_{yz}^N(x)= \varphi^N_{yz}(x)$. Thus, we 
also have
\[
\varphi^N(x)= \sum_{\substack{{(y,z)\in A}\\{\varphi^{N}_{yz}>0}}} \frac{\varphi^N_{yz}(x)}{\varphi^N_{yz}}=
\sum_{\substack{{(y,z)\in A}\\{\varphi^{N}_{yz}>0}}} \frac{\lambda^N_{yz}(x)}{\varphi^N_{yz}}=  \lambda^N(x).
\]
Moreover, by Proposition \ref{ineq-utile}, it is immediately deduced that $\varphi^N(X)\leq \lambda^N(X).$
Finally, Proposition \ref{bang} shows that when $|X|\geq 2$ we generally have $\varphi^N(X)\neq \lambda^N(X).$

We finally observe that, from the computational point of view, there is an important difference between $\varphi$ and  $\lambda$. Indeed, in order to compute the term $\varphi^N_{yz}(X)=\varphi^N_{yz}-\varphi^{N_X}_{yz}$, it is enough to dispose just of a single maximum flow  from $y$ to $z$ in $N$ and a single maximum flow  from $y$ to $z$ in $N_X$. On the other hand, in order to compute the term $\lambda^N_{yz}(X)$, one needs, in principle, to know the decompositions of all the possible maximum flows from $y$ to $Z$ in $N$.

\section{Conclusions and further research}\label{citation} 

We have clarified the equality between two fundamental concepts in network theory that surprisingly were identified without a solid  formal argument: the minimum number $\lambda_{yz}^N(x)$ of paths passing through the  vertex $x$ in a maximum sequence  of  arc-disjoint  paths from $y$ to $z$ in the network $N$ and the number $\varphi^N_{yz}(x)$ expressing the falling of maximum flow value from $y$ to  $z$ in $N$ when the capacity of all the arcs incident to $x$ is reduced to zero. 
The proof of that fact has involved a tricky analysis of the relationship between paths and flows which goes beyond the original goal and is interesting in itself.
We also proved that the natural generalizations of $\lambda_{yz}^N(x)$ and $\varphi_{yz}^N(x)$ to sets  of vertices are not, in general,  equal. 

Our analysis has led to the definition of two conceptually different centrality measures $\varphi$ and $\lambda$. The contexts in which they could fruitfully applied are in principle many. Just to give an example, consider a scientific community $V$ of scholars and, for every $x,y\in V$, the number $c(x,y)$ of times that, within a certain fixed period of time, the researcher $x\in V$ cited the researcher $y\in V$. Construct then the corresponding citation network
$N=(K_V,c)\in\mathscr{N}$. It is reasonable to think that, for $x,y,z\in V$, if $x$ cited $y$ and $y$ cited $z$, then indirectly $x$ cited $z$ so that $z$ gains  prestige not only from $y$ but also from $x$ and thus, more generally, $z$ gains prestige from any path in $N$ having $z$ as endpoint.\footnote{A similar approach by path consideration is used for sport competitions in Bubboloni and Gori (2018) with the scope to obtain a ranking of teams.}
Consider, in particular, a situation in which two groups of researchers, say $X_1$ and $X_2$, received the same total amount of citations from the scholars outside the groups, that is, $\sum_{a\in A_{X_{1}}^-}c(a)=\sum_{a\in A_{X_{2}}^-}c(a)$.
 Suppose that you want to diversify $X_1$ and $X_2$ putting in evidence the quality of those citations. For $X_i$, where $i\in\{1,2\}$, that quality can be evaluated by looking at the set $Y_i$ of  scholars who cited the scholars in $X_i$ and taking into account the number of citations that the scholars in $Y_i$ themselves received. On the other hand, also the quality of the citations received by the scholars in $Y_i$ is important and that can be in turn evaluated by looking at the set $Z_i$ of the scholars who cited the scholars in $Y_i$. That reasoning can be continued so that the quality of the citations received by the scholars in $X_i$ can effectively emerge only  by a global approach which takes into consideration the configuration and complexity of the whole citation network  (see, for instance Bouyssou and Marchant, 2016).
That can operatively be committed exactly by computing which between $\lambda^N(X_1)$ and $\lambda^N(X_2)$ is larger.  
 Another reasonable idea could be instead to
 look at the impact on the amount of direct and indirect citations in the network caused by a hypothetical absence from the scientific scenario of  the scholars in $X_i$, that is, comparing  $\varphi^N(X_1)$ and $\varphi^N(X_2)$.  Indeed, if $\varphi^N(X_1)>\varphi^N(X_2)$, then the absence of $X_1$ mostly damages the scientific community in its dynamic exchange of contacts,  with the number of global citations becoming poorer. 
 
We also emphasize that the formal approach we used seems very promising for dealing with the properties of our centrality measures, in particular, with those invoked by Sabidussi (1966) as the main desirable. As well-known an axiomatic satisfactory definition of centrality is missing in the literature. The presence or absence of certain properties can though help in deciding which measure better fit in a certain application, as largely recognized in the bibliometric literature which is recently oriented in using methods from social choice theory (see for instance Chambers and Miller, 2014; Csat\'o, 2019).

Many other aspects surely deserve to be deepen and will constitute further lines of research: computational aspects, construction of algorithms for the effective calculation of $\varphi$ and $\lambda$ and, of course, implementation of them on concrete networks in order to discover in which sense and in which kind of networks, our centrality measures behave better than other classic centrality measures.

\vspace{7mm}

\noindent {{\bf Acknowledgments}} We wish to thank J\o rgen Bang-Jensen for his illuminating example which led to Proposition \ref{bang}, and for
 the kind permission to use his personal communication. We also thank three anonymous referees whose advice helped in improving the readability of the paper. Daniela Bubboloni is partially supported by GNSAGA of INdAM (Italy).

\vspace{7mm}

\noindent {\Large{\bf References}}

\vspace{3mm}

\noindent Ahuja, R.K., Magnanti, T.L., Orlin J.B., 1993.  Network Flows. Prentice Hall, Englewood Cliffs, NJ.
\vspace{2mm}

\noindent Bang-Jensen, J., Gutin, G., 2008. Digraphs Theory, Algorithms and Applications. Springer.
\vspace{2mm}

\noindent Bang-Jensen, J., 2019. Private communication.
\vspace{2mm}



\noindent  Borgatti, S.P., Everett, M.G., 2006. A graph-theoretic framework for classifying centrality measures. Social Networks 28, 466-484.
\vspace{2mm}

\noindent  Borgatti, S.P., Everett, M.G., Freeman, L.C. 2002. Ucinet $6$ for Windows: Software for Social Network Analysis. Harvard, MA: Analytic Technologies.
\vspace{2mm}

\noindent  Bouyssou, D., Marchant, T., 2016. Ranking authors using fractional counting of citations: An axiomatic approach. Journal of Informetrics 10, 183-199.
\vspace{2mm}

\noindent  Bubboloni, D, Gori, M., 2018. The flow network method. Social Choice and Welfare 51, 621-656.
\vspace{2mm}

\noindent Chambers, C.P., Miller, A.D., 2014.  Scholarly influence. Journal of Economic Theory 151, 571-583.
\vspace{2mm}

\noindent Csat\'o, L., 2019. Journal ranking should depend on the level of aggregation. Journal of Informetrics 13, https://doi.org/10.1016/j.joi.2019.100975.
\vspace{2mm}

\noindent Eboli, M., 2019. A flow network analysis of direct balance-sheet contagion in financial networks. Journal of Economic Dynamics \& Control 103, 205-233.
\vspace{2mm}

\noindent   Everett, M.G., Borgatti, S.P., 1999. The centrality of groups and classes. Journal of Mathematical Sociology 23, 181-201.
\vspace{2mm}

\noindent  Everett, M.G., Borgatti, S.P., 2005. Extending Centrality. In P. Carrington, J. Scott, \& S. Wasserman (Eds.), Models and Methods in Social Network Analysis (Structural Analysis in the Social Sciences, pp. 57-76). Cambridge, Cambridge University Press.
\vspace{2mm}

\noindent  Ford, L.R., Fulkerson, D.R., 1956. Maximal flow through a network. Canadian Journal of Mathematics 8, 399-404.
\vspace{2mm}

\noindent  Freeman, L., 1979. Centrality in networks: I. Conceptual clarification. Social networks 1, 215-239.
\vspace{2mm}

\noindent Freeman, L.,  Borgatti, S. P., White, D. R., 1991. Centrality in valued graphs: a measure of betweenness based on network flow. Social Networks 13, 141-154.
\vspace{2mm}

\noindent Ghiggi, S., 2018. Flow centrality on networks, Master Thesis (supervisor D. Bubboloni), Department of Mathematics and Informatics, University of Florence (Italy).
\vspace{2mm}

\noindent  G\'omez, D., Figueira J.R., Eus\'ebio, A., 2013.  Modeling centrality measures in social network analysis using bi-criteria network flow optimization problems. European Journal of Operational Research 226, 354-365.
\vspace{2mm}

\noindent  Kosch\"utzki, D., Lehmann, K.A., Peeters, L., Richter, S., Tenfelde-Podehl,
D., Zlotowski, O., 2005. Centrality Indices. Chapter 3 in Brandes and Erlebach Network Analysis: Methodological Foundations, Volume 3418 of LNCS Tutorial, Springer.
\vspace{2mm}

\noindent  Newman, M., 2005. A measure of betweenness centrality based on random walks. Social Networks 27, 39-54.
\vspace{2mm}

\noindent  R Development Core Team. 2007.  R: A language and environment for Statistics and Computing. R foundation for Statistical Computing, Vienna, Austria. ISBN 3-900051-07-0, Version 2.6.1, URL http://www.R-project.org/. 
\vspace{2mm}

\noindent  Sabidussi, G., 1966. The centrality index of a graph. Psychometrika 31, 581-603.
\vspace{2mm}

\noindent Trimponias, G., Xiao, Y., Xu, H., Wu, X., Geng, Y.,  2017. Node-Constrained Traffic Engineering: Theory and Applications.  IEEE/ACM Transactions on Networking  27, 1344-1358
\vspace{2mm}

\noindent  Zweig, K., A., 2016.  Analysis Literacy. Springer.
\vspace{2mm}

\end{document}